\documentclass[reqno,12pt]{amsart}
 \usepackage{amssymb}
 \usepackage{latexsym}
 \usepackage{graphicx}
 \usepackage{multicol}
 \usepackage{mathrsfs}
 \usepackage{young}
 \usepackage[vcentermath,enableskew,stdtext]{youngtab}
 \usepackage[left=1.11in,right=1.11in,top=1.1in,bottom=1.14in]{geometry}
 \usepackage[all]{xy}
    \SelectTips{cm}{10}     
    \everyxy={<2.5em,0em>:} 
 \usepackage{fancyhdr}      
 \usepackage{multicol}
 \usepackage{multirow}
 \usepackage{cases}
 \usepackage{stmaryrd, bm, enumitem}

\newcounter{ctr}
\theoremstyle{plain}
\newtheorem{theorem}{Theorem}[section]

\newtheorem*{lemma*}{Lemma}
\newtheorem{lemma}[theorem]{Lemma}
\newtheorem{corollary}[theorem]{Corollary}
\newtheorem{proposition}[theorem]{Proposition}

\theoremstyle{definition}
\newtheorem{definition}[theorem]{Definition}

\newtheorem{remark}[theorem]{Remark}
\newtheorem{example}[theorem]{Example}

\newcommand{\aut}{\text{\rm Aut}\,}

\newcommand{\Cbasis}{\ensuremath{\mathcal{C}}}
\newcommand{\CC}{\ensuremath{\mathbb{C}}}

\newcommand{\End}{\text{\rm End}}

\newcommand{\FF}{\ensuremath{\mathbb{F}}}

\renewcommand{\H}{\ensuremath{\mathscr{H}}}

\renewcommand{\hom}{\text{\rm Hom}}

\newcommand{\idelm}{\ensuremath{id}}

\newcommand{\mat}[1]{\ensuremath{\left( %
        \begin{array}{cccccccccccccccccccccccccccc} #1 \end{array}\right)}}

\newcommand{\QQ}{\ensuremath{\mathbb{Q}}}
\DeclareMathOperator{\rad}{rad}

\newcommand{\RR}{\ensuremath{\mathbb{R}}}

\newcommand{\ZZ}{\ensuremath{\mathbb{Z}}}

\newcommand{\be}{\begin{equation}}
\newcommand{\ee}{\end{equation}}
\newcommand{\Res}{\text{\rm Res}}
\renewcommand{\S}{\ensuremath{\mathcal{S}}}
\newcommand{\tsr}{\ensuremath{\otimes}}
\newcommand{\C}{\ensuremath{C^{\prime}}} 

\newcommand{\nsbr}[1]{{\ensuremath{\check{#1}}}}

\renewcommand{\u}{\ensuremath{u}}  
\newcommand{\ui}{\ensuremath{u^{-1}}} 

\newcommand{\Tab}{\ensuremath{\mathcal{T}}}

\renewcommand{\ng}{\text{-}}

\newcommand{\pnp}{\ensuremath{P \neq NP}}
\newcommand{\pvnp}{\ensuremath{P \text{ vs. } NP}}

\newcommand{\nsH}{\nsbr{\H}}

\newcommand{\two}{[2]}
\newcommand{\sP}{\mathcal{P}}
\newcommand{\sQ}{\mathcal{Q}}
\newcommand{\field}{\ensuremath{K}}

\newcommand{\f}{\ensuremath{f}}

\newcommand{\thetaalg}{\mathscr{A}_\theta}
\newcommand{\dualtheta}{\ensuremath{\#}}
\newcommand{\dualone}{\ensuremath{\diamond}}
\newcommand{\compositions}[2]{\ensuremath{\text{SComp}_{#1,#2}}}

\newlength{\cellsize}
\newcommand\tableau[1]{
\vcenter{
\let\\=\cr
\baselineskip=-16000pt \lineskiplimit=16000pt \lineskip=0pt
\halign{&\tableaucell{##}\cr#1\crcr}}}

\newcommand{\tableaucell}[1]{{%
\def \arg{#1}\def \void{}%
\ifx \void \arg
\vbox to \cellsize{\vfil \hrule width \cellsize height 0pt}%
\else \unitlength=\cellsize
\begin{picture}(1,1)
\put(0,0){\makebox(1,1){$#1$}} \put(0,0){\line(1,0){1}}
\put(0,1){\line(1,0){1}} \put(0,0){\line(0,1){1}}
\put(1,0){\line(0,1){1}}
\end{picture}%
\fi}}

\author{Jonah Blasiak}\thanks{The author was supported by an NSF postdoctoral fellowship.}
\title{Nonstandard braid relations and Chebyshev polynomials}

\address{Department of Mathematics, Drexel University, Philadelphia, PA 19104}
\email{jblasiak@gmail.com} 

\begin{document}
\maketitle
\begin{abstract}
A fundamental open problem in algebraic combinatorics is to find a positive combinatorial formula for Kronecker coefficients, which are multiplicities of the decomposition of the tensor product of two  $\S_r$-irreducibles into irreducibles.  Mulmuley and Sohoni attempt to solve this problem using canonical basis theory, by first constructing a nonstandard Hecke algebra $B_r$, which, though not a Hopf algebra, is a $\u$-analogue of the Hopf algebra $\CC \S_r$ in some sense (where $\u$ is the Hecke algebra parameter).  For  $r=3,$ we study this Hopf-like structure in detail. We define a nonstandard
Hecke algebra $\nsH^{(k)}_3 \subseteq \H_3^{\tsr k}$, determine its irreducible representations over $\QQ(\u)$, and show that it has a presentation with a nonstandard braid relation that involves Chebyshev polynomials evaluated at  $\frac{1}{\u + \ui}$.  We generalize this to Hecke algebras of dihedral groups.  We go on to show that these nonstandard Hecke algebras have bases similar to the Kazhdan-Lusztig basis of $\H_3$ and are cellular algebras in the sense of Graham and Lehrer.
\end{abstract}

\section{Introduction}
\label{s Introduction}
Let  $\S_r$ denote the symmetric group on $r$ letters and let $M_\nu$ be the $\S_r$-irreducible corresponding to the partition $\nu$.
The Kronecker coefficient $g_{\lambda \mu \nu}$ is the multiplicity of $M_\nu$ in the tensor product $M_\lambda \otimes M_\mu$.  A fundamental and difficult open problem in algebraic combinatorics is to find a positive combinatorial formula for these coefficients. Although this problem has been studied since the early twentieth century, the general case still seems out of reach. In the last ten years this problem has seen a resurgence of effort, perhaps because of its recently discovered connections to quantum information theory \cite{BWZ} and complexity theory \cite{GCT4}.
Much of the recent progress has been  for Kronecker coefficients indexed by two two-row shapes, i.e., when $\lambda$ and $\mu$ have two rows: an explicit, though not positive, formula was given by Remmel and Whitehead in \cite{RW} and further improvements were made by Rosas \cite{Rosas} and Briand-Orellana-Rosas \cite{BOR}.  Briand-Orellana-Rosas \cite{BOR, BOR2} and Ballantine-Orellana \cite{BO} have also made progress on the special case of reduced Kronecker coefficients, sometimes called the stable limit, in which the first part of the partitions $\lambda,\mu,\nu$ is large.

In a series of recent papers, Mulmuley, in part with Sohoni and Narayanan, describes an approach to $\pvnp$ and related lower bound problems in complexity theory using algebraic geometry and representation theory, termed geometric complexity theory.  Understanding Kronecker coefficients, particularly, having a good rule for when they are zero, is critical to their plan. In fact, Mulmuley gives a substantial informal argument claiming that if certain difficult separation conjectures like $\pnp$ are  true, then there is a $\# P$ formula for Kronecker coefficients and a polynomial time algorithm that determines whether a Kronecker coefficient is nonzero \cite{GCT6}.  Thus from the complexity-theoretic perspective, there is hope that Kronecker coefficients will have nice formulae like those for Littlewood-Richardson coefficients, though experience suggests they will be much harder.

A useful perspective for studying tensor products of  $\S_r$-modules is to endow the group algebra  $\ZZ\S_r$ with the structure of a Hopf algebra.  The coproduct is
$\Delta :\ZZ \S_r \to \ZZ \S_r \tsr \ZZ \S_r$,  $w \mapsto w \tsr w$, and the $\ZZ\S_r$-module  $M_\lambda \tsr M_\mu$ is then defined to be the restriction of the $\ZZ\S_r \tsr \ZZ \S_r$-module  $M_\lambda \boxtimes M_\mu$ along  $\Delta$.

In \cite{GCT4}, Mulmuley and Sohoni attempt to use canonical bases to understand Kronecker coefficients by constructing an algebra defined over $\ZZ[\u,\ui]$ that carries some of the information of the Hopf algebra $\ZZ \S_r$ and specializes to it at $\u=1$.  Specifically, they construct the nonstandard Hecke algebra  $\nsH_r$ (denoted  $B_r$ in \cite{GCT4}), which is a subalgebra of the tensor square of the Hecke algebra $\H_r$ such that the inclusion $\nsbr{\Delta} : \nsH_r \hookrightarrow \H_r \tsr \H_r$ is a $\u$-analogue of the coproduct  $\Delta$ of $\ZZ \S_r$ (see Definition \ref{d nonstandard Hecke algebra}).  The goal is then to break up the Kronecker problem into two steps \cite{GCT7}:
\begin{enumerate}
\item Determine the multiplicity $n_{\lambda,\mu}^\alpha $ of an irreducible $\nsH_r$-module  $\nsbr{M}_\alpha$ in the tensor product $M_\lambda \otimes M_\mu$.
\item Determine the multiplicity $m^\nu_\alpha$ of the  $\S_r$-irreducible  $M_\nu$ in $\nsbr{M}_\alpha|_{\u =1}$.
\end{enumerate}
The resulting formula for Kronecker coefficients is
\be g_{\lambda \mu \nu}= \sum_\alpha n_{\lambda,\mu}^\alpha m^\nu_\alpha.
\ee
Thus a positive combinatorial formula for $n_{\lambda,\mu}^\alpha$ and $m^\nu_\alpha$ would yield  one for Kronecker coefficients.

However, this approach meets with serious difficulties.  The defining relations of the algebras $\nsH_r$ seem to be extremely complicated and remain mysterious even for $r=4$. Problem (1) seems to be within reach, and, in the forthcoming paper \cite{B4}, we solve it in the two-row case. For problem (2), the hope is to find a canonical basis of $\nsbr{M}_\alpha $ that has a cellular decomposition into  $\S_r$-irreducibles at  $\u=1$, however this seems to be extremely difficult.

In this paper we study a family of algebras $\nsH^{(k)}_3$ that contains  $\H_3$ and  $\nsH_3$ as cases  $k=1,2$.  We discover a remarkable connection between the defining relations of these algebras and Chebyshev polynomials $T_k(x)$.  Specifically, we show that  $\nsH^{(k)}_3$ is generated by $\sP_1^{(k)}, \sP_2^{(k)}$ and has a relation, which we call the nonstandard braid relation, that generalizes the braid relation for $k=1$:
\begin{multline}
\label{e nonstandard braid relation 0}
\sP_1^{(k)}
(\sP^{(k)}_{21} - (\two^{k} a^{(1)})^2) (\sP^{(k)}_{21} - (\two^{k} a^{(2)})^2) \dots (\sP^{(k)}_{21} - ([2]^k a^{(k)})^2)= \\
\sP_2^{(k)}
(\sP^{(k)}_{12} - (\two^{k} a^{(1)})^2) (\sP^{(k)}_{12} - (\two^{k} a^{(2)})^2) \dots (\sP^{(k)}_{12} - ([2]^k a^{(k)})^2),
\end{multline}
where $\two = \u+\ui$, $\sP^{(k)}_{i_1 i_2} = \sP^{(k)}_{i_1} \sP^{(k)}_{i_2}$, and the coefficient $a^{(j)}$ is equal to $T_j(\frac{1}{\two}).$

Chebyshev polynomials come up in several places in nearby areas of algebra, however their appearance here seems to be genuinely new. For example, Chebyshev polynomials appear in the criterion for semisimplicity of Temperley-Lieb and Jones algebras \cite{GHJ, GrahamLehrer} (also see \cite{BM}).  In \cite{BM} they appear in three ways---as just mentioned, in giving the dimension of a centralizer algebra of a Temperley-Lieb algebra, and in calculating the decomposition of a Brauer algebra module into Temperley-Lieb algebra modules.  In this paper, Chebyshev polynomials appear as the coefficients in the relations of an algebra.

This paper is organized as follows.  In Section \ref{s Nonstandard Hecke Algebras} we define the nonstandard Hecke algebras $\nsH^{(k)}_W$ for any Coxeter group  $W$ and establish some of their basic properties.  Section \ref{s A generalization of the braid relation} contains our main results---a description of the irreducible representations of $\nsH^{(k)}_3$ (Theorem \ref{t nsH S3 2D reps}) and a presentation for  $\nsH^{(k)}_3$ (Theorem \ref{t S3 coprod k relations}).
In Section \ref{s Generalizations to dihedral groups} we generalize these results to nonstandard Hecke algebras  $\nsH^{(k)}_W$ of dihedral groups.  In this case, the nonstandard braid relation involves a multivariate version of Chebyshev polynomials and the $\nsH^{(k)}_W$-irreducibles are parameterized by signed compositions of  $k$ (see Definition \ref{d signed compositions}).
We further show (Section \ref{s cellular basis}) that the nonstandard Hecke algebras   of dihedral groups have bases generalizing the Kazhdan-Lusztig basis of $\H_3$ and are cellular algebras in the sense of Graham and Lehrer \cite{GrahamLehrer}.


\section{Nonstandard Hecke Algebras}
\label{s Nonstandard Hecke Algebras}
After recalling the definition of the (standard) Hecke algebra  $\H_W$ of a Coxeter group  $W$, we introduce the nonstandard Hecke algebra  $\nsH_W$ of \cite{GCT4}.  Hecke algebras are not Hopf algebras in a natural way, and the nonstandard Hecke algebra  $\nsH_W$ is in a sense the smallest deformation of $\H_W$ that also deforms the Hopf algebra structure of the group algebra  $\ZZ W$. We show that the Hecke algebra $R \H_{\S_2}$ is a Hopf algebra (for suitable rings  $R$).  We then define the sequence of algebras $\nsH^{(k)}_W$, $k \geq 1$, that begins with $\H_W$ ($k=1$) and $\nsH_W$ ($k=2$).  We record some basic facts about the representation theory of these algebras and define anti-automorphisms that behave like the antipode of a Hopf algebra.

\subsection{}
\label{ss coxeter groups}
Let $(W,S)$ be a Coxeter group with length function $\ell$.
We work over the ground ring $\mathbf{A} = \ZZ[\u,\ui]$, the ring of Laurent polynomials in the indeterminate $\u$.

\begin{definition}
The \emph{Hecke algebra} $\H_W$ of a Coxeter group $(W, S)$ is the free $\mathbf{A}$-module with basis $\{T_w :\ w\in W\}$ and relations generated by
\begin{equation}\label{hecke eq}
\begin{array}{ll}T_uT_v = T_{uv} & \text{if } \ell(uv)=\ell(u)+\ell(v)\\
 (T_s - \u)(T_s + \ui) = 0 & \text{if } s\in S.\end{array}
\end{equation}
\end{definition}

For any $J\subseteq S$, the \emph{parabolic subgroup} $W_J$ is the subgroup of $W$ generated by $J$.
We let $(\H_W)_J$ denote the subalgebra of $\H_W$ with $\mathbf{A}$-basis $\{T_w:\ w\in W_J\}$, which is isomorphic to $\H_{W_J}$.

For any commutative ring $\field$ and ring homomorphism $\mathbf{A} \to \field$, let $\field \H_W = \field \tsr_{\mathbf{A}} \H_W$.  We will often let $\field = \QQ(u),$ the quotient field of $\mathbf{A}$.  If the ring $\field$ is understood and $\mathbf{A} \to \field$ is given by $\u \mapsto z$, then we also write $\H_W|_{\u = z}$ for $\field \tsr_\mathbf{A} \H_W$.  The Hecke algebra $\H_W$ over $\mathbf{A}$ is the \emph{generic Hecke algebra of  $W$} and $\field \H_W$ is a \emph{specialization} of $\H_W$.

We are particularly interested in the type $A$ case in this paper, and in this case  $(W,S) = (\S_r, \{s_1,\ldots,s_{r-1}\})$ and we abbreviate $\H_{\S_r}$ by $\H_r$.

\subsection{}
\label{ss theta involution}
The $\u$-integers are $[k] := \frac{\u^k - \u^{-k}}{\u - \ui} \in \mathbf{A}$. We also use the notation  $[k]$ to denote the set $\{1,\ldots,k\}$, but these usages should be easy to distinguish from context. We also set $\f = [2]^2$ because this constant appears particularly often.

Let $\C_s = T_s + \ui$ and  $C_s = T_s - \u$ for each  $s \in S$.  These are the simplest unsigned and signed Kazhdan-Lusztig basis elements.  They are also proportional to the primitive central idempotents of  $\field (\H_W)_{\{s\}} \cong \field \H_2$, provided the constant $[2]$ is invertible in $\field$;  define $p_+ = \frac{1}{[2]} \C_{s_1}$ (respectively $p_- = - \frac{1}{[2]} C_{s_1}$) to be the idempotent corresponding to the trivial (respectively sign) representation of $\field \H_2$.

Write $\epsilon_+, \epsilon_-$ for the one-dimensional trivial and sign representations of $\H_W$, which are defined by
\[ \begin{array}{cccc}
  \epsilon_+ : \C_s \mapsto \two, && \epsilon_- : \C_s \mapsto 0, & s \in S.
\end{array} \]
We identify these algebra homomorphisms $\epsilon_+, \epsilon_-: \H_W \to \mathbf{A}$ with left $\H_W$-modules in the usual way.


There is an $\mathbf{A}$-algebra automorphism $\theta : \H_W \to \H_W$ defined by $\theta(T_s) = - T_s^{-1},\ s \in S$. Let $1^{\text{op}}$ be the $\mathbf{A}$-anti-automorphism of $\H_W$ given by $1^{\text{op}}(T_w) = T_{w^{-1}}$.
Let $\theta^{\text{op}}$ be the $\mathbf{A}$-anti-automorphism of $\H_W$ given by $\theta^{\text{op}} = \theta \circ 1^{\text{op}} = 1^{\text{op}} \circ \theta$.  We will establish some basic properties of these anti-automorphisms in \textsection\ref{ss theta on nsH k reps}.

Let  $\eta$ be the unique $\mathbf{A}$-algebra homomorphism from  $\mathbf{A}$ to $\H_W$.
At  $\u=1$, the maps $\eta, \epsilon_+, 1^\text{op}$ specialize to the unit, counit, and antipode of the Hopf algebra  $\ZZ W$.

\subsection{}
\label{ss nonstandard Hecke algebra definition}
Here we introduce the nonstandard Hecke algebra  $\nsH_W$ from \cite{GCT4} (denoted  $B_r$ there in the case  $W = \S_r$), and show that $R \nsH_2$ is isomorphic to $R \H_2$ (for suitable $R$), thereby giving a Hopf algebra structure on $R \H_2$.  We also show that the anti-automorphism $1^\text{op}$ behaves like an antipode of the Hopf algebra-like object $\nsH_W$.

\begin{definition}
\label{d nonstandard Hecke algebra}
The \emph{nonstandard Hecke algebra}  $\nsH_W$ is the subalgebra of $\H_W \tsr \H_W$ generated by the elements
\[ \sP_s := \C_s \tsr \C_s + C_s \tsr C_s, \ s \in S. \]
We let $\nsbr{\Delta}:\nsH_W \hookrightarrow   \H_W \tsr \H_W$ denote the canonical inclusion, which we think of as a deformation of the coproduct $\Delta_{\ZZ W} :\ZZ W \to \ZZ W \tsr \ZZ W$,  $w \mapsto w \tsr w$.
\end{definition}
The nonstandard Hecke algebra is also the subalgebra of $\H_W \tsr \H_W$ generated by
\[ \sQ_s := \f - \sP_s = - \C_s \tsr C_s - C_s \tsr \C_s, \ s \in S. \]
Despite their simple definition, the nonstandard Hecke algebras seem to be extremely difficult to describe in terms of generators and relations.  Indeed, the main purpose of this paper is to work out such a presentation for dihedral groups $W$.
For the easiest case $W = \S_2$, the story is quite nice.

\begin{proposition}
\label{p H2 Hopf algebra}
Set  $R = \mathbf{A}[\frac{1}{[2]}]$ and  $R_1 = \ZZ[\frac{1}{2}]$ (so that  $R_1$ is the  $\u=1$ specialization of  $R$).  We have $R \nsH_2 \cong R\H_2$ by $\frac{1}{\f} \sP_1 \mapsto p_+$. Then
\begin{list}{\emph{(\roman{ctr})}} {\usecounter{ctr} \setlength{\itemsep}{1pt} \setlength{\topsep}{2pt}}
\item $R \H_2$ is a Hopf algebra with coproduct $\Delta =\nsbr{\Delta}$, antipode $1^{\text{op}}$, counit $\epsilon_+$, and unit $\eta$.
\item the Hopf algebra $R \H_2|_{\u = 1}$, with Hopf algebra structure coming from (i), is isomorphic to the group algebra $R_1 \S_2$ with its usual Hopf algebra structure.
\end{list}
Moreover, the Hopf algebra structure of (i) is the unique way to make the algebra $R\H_2$ into a Hopf algebra so that (ii) is satisfied.
\end{proposition}

\begin{proof}
The isomorphism $R \nsH_2 \cong R\H_2$ is immediate from the observation that $\nsbr{\Delta} (\frac{1}{\f} \sP_s) = p_+ \tsr p_+ + p_- \tsr p_-$ is an idempotent in  $R \H_2 \tsr R \H_2$.

The axiom for the antipode is a special case of Proposition \ref{p hecke algebra antipode} to come (and also easy to check directly). We need to check the axioms
\[ (\epsilon_+ \tsr \idelm) \circ \Delta = \idelm = (\idelm \tsr \epsilon_+) \circ \Delta, \]
\[ (\Delta \tsr \idelm) \circ \Delta = (\idelm \tsr \Delta) \circ \Delta, \]
which is easy. For example, for the second, observe that both sides applied to $p_+$ yield
\[ p_+ \tsr p_+ \tsr p_+ + p_+ \tsr p_- \tsr p_- + p_- \tsr p_+ \tsr p_- + p_- \tsr p_- \tsr p_+. \]

It is straightforward to see that (ii) is satisfied.  We check only that the coproduct commutes with specialization, and omit verifying that this is also so for the antipode, the counit, and the unit:
\begin{multline*}
\Delta|_{\u =1}(\frac{1}{2} (1 + s_1)) = \frac{1}{2} (1 + s_1) \tsr \frac{1}{2} (1 + s_1) + \frac{1}{2} (1 - s_1) \tsr \frac{1}{2} (1 - s_1) \\
= \frac{1}{2} (1 \tsr 1 + s_1 \tsr s_1) = \Delta_{R_1 \S_2}(\frac{1}{2} (1 + s_1)), \end{multline*}

where  $\Delta|_{\u =1}$ is the specialization of $\Delta $ and $\Delta_{R_1 \S_2}$ is the usual coproduct on  $R_1 \S_2$.

For the uniqueness statement, we use that $ R \H_2 \tsr R \H_2$ is isomorphic to a product of matrix algebras. Explicitly,
\[ R \H_2 \tsr R \H_2 \cong R (p_+ \tsr p_+) \oplus R (p_+ \tsr p_-) \oplus R (p_- \tsr p_+) \oplus R (p_- \tsr p_-) . \]
The map $\Delta$ is determined by  $\Delta(p_+)$, and the image of  $\Delta$ is isomorphic to  $R \H_2$ if and only if  $\Delta(p_+)$ is an idempotent not equal to the identity.  We also have that $\Delta_{R_1 \S_2}(p_+) = p_+ \tsr p_+ + p_- \tsr p_-$.  The only idempotent of $R \H_2 \tsr R \H_2$ that specializes to $\Delta_{R_1 \S_2}(p_+)$ at  $\u = 1$ is $p_+ \tsr p_+ + p_- \tsr p_-$, hence this must be  $\Delta(p_+)$ as desired.
Additionally, the counit is determined uniquely by the comultiplication; the only anti-automorphisms of $R \H_2$ are  $1^\text{op}$ and  $\theta^{\text{op}}$, and only $1^\text{op}$ satisfies the required axiom.
\end{proof}
%

Analogous to the trivial and sign representations of  $\H_W$, there are one-dimensional trivial and sign representations of $\nsH_W$, which we denote by $\nsbr{\epsilon}_{+}$ and $\nsbr{\epsilon}_{-}$:
\[ \begin{array}{cccc}
  \nsbr{\epsilon}_{+} : \sP_{s} \mapsto \two^{2}, && \nsbr{\epsilon}_{-} : \sP_s \mapsto 0, & s \in S.
\end{array} \]

For the next proposition, let $\eta,1^{\text{op}},\theta^\text{op}$ be as in \textsection\ref{ss theta involution}, and let $\mu$ be the multiplication map for $\H_W$.

\begin{proposition} \label{p hecke algebra antipode}
The involutions $1^{\text{op}}$ and $\theta^{\text{op}}$ are antipodes in the following sense:
\begin{flalign}
\label{e antipode 1}
\mu \circ (1^{\text{op}} \tsr 1) \circ \nsbr{\Delta} &= \eta \circ \nsbr{\epsilon}_+,  \\
\label{e antipode 2}
\mu \circ (\theta^{\text{op}} \tsr 1) \circ \nsbr{\Delta} &= \eta \circ \nsbr{\epsilon}_-,
\end{flalign}
where these are equalities of maps from $\nsH_W$ to $\H_W$.
\end{proposition}
\begin{proof}
The right-hand side of (\ref{e antipode 1}) is the algebra homomorphism defined by $\eta \circ \nsbr{\epsilon}_+(\sQ_s) = 0$, $s \in S$. This is a linear map from  $\nsH_W$ to  $\H_W$ which sends the two-sided ideal generated by the  $\sQ_s$ to 0 and sends 1 to 1, and there is only one linear map with these properties.  To see that the left-hand side is also this linear map, first observe that it takes $\sQ_s$ to $0$ and $1$ to $1$:
\be
\label{e antipode on Qs}
\mu \circ (1^{\text{op}} \tsr 1) (\sQ_s) = \mu(\sQ_s) = - \C_s C_s - C_s \C_s = 0,\ s \in S.
\ee
Next, let $\nsbr{\Delta}(\sQ_s) = \sum_i a^{(i)} \tsr b^{(i)}$ and $\nsbr{\Delta}(x) = \sum_j c^{(j)} \tsr d^{(j)}$ for some $x \in \nsH_W$. Since $\nsbr{\Delta}$ is an algebra homomorphism, there holds
\be
 \begin{array}{cll}
 \mu \circ (1^{\text{op}} \tsr 1) \circ \nsbr{\Delta}(\sQ_s x) &=& \mu \circ (1^{\text{op}} \tsr 1) \left( \sum_i (a^{(i)} \tsr b^{(i)}) \sum_j (c^{(j)} \tsr d^{(j)}) \right) \\
 &=& \sum_{i, j} 1^{\text{op}}(c^{(j)}) 1^{\text{op}}(a^{(i)}) b^{(i)} d^{(j)} \\
 &=& 0,
 \end{array}
 \ee
where the last equality is by (\ref{e antipode on Qs}).

We can similarly show that $\mu \circ (1^{\text{op}} \tsr 1) \circ \nsbr{\Delta}(x \sQ_s) = 0$ for any $x \in \nsH_W$. Thus the two-sided ideal generated by the $\sQ_s$ is sent to zero, so the left and right-hand sides of (\ref{e antipode 1}) agree. Equation (\ref{e antipode 2}) is proved in a similar way by replacing $\sQ_s$ with $\sP_s$ and $\epsilon_+$ with $\epsilon_-$ above.
\end{proof}

\subsection{}
By Proposition \ref{p H2 Hopf algebra}, we can define $\Delta^{(k)} : R \H_2 \to R \H_2^{\tsr k}$ inductively by $\Delta^{(k)} := (\Delta^{(k-1)} \tsr 1) \circ \Delta = (1 \tsr \Delta^{(k-1)}) \circ \Delta$ and $\Delta^{(2)} := \Delta$. Explicitly,
\be
\label{e explicit coproduct}
\Delta^{(k)}(p_+) = \sum_{\substack{ \mathbf{a} \in \{ +, - \}^{k} \\ |\{ i : a_i = - \}| \text{ is even}} } p_{a_1} \tsr \dots \tsr p_{a_k}.
\ee
It is now natural to generalize $\nsH_{W}$ as follows.
\begin{definition}
The \emph{nonstandard Hecke algebra} $\nsH^{(k)}_W$ is the subalgebra of $\H_W^{\tsr k}$ generated by the $\sP^{(k)}_s := \two^{k} \Delta_s^{(k)}(p_{+})$ for all $s \in S$, where $\Delta_s^{(k)} = \iota_s^{\tsr k} \circ \Delta^{(k)}$ and $\iota_s$  is the inclusion $\H_2 \cong (\H_{W})_{\{s\}} \hookrightarrow \H_W$.  Let $\nsbr{\Delta}^{(1^{k})} : \nsH^{(k)}_{W} \hookrightarrow \H_{W}^{\tsr k}$ be the canonical inclusion.
\end{definition}

Set $\sQ^{(k)}_s = [2]^k - \sP^{(k)}_s$.  The set $S$ of simple reflections of $ W$ will always be denoted $\{s_1,\ldots,s_{r-1}\}$ in this paper, and we let $\sP^{(k)}_{i_1 i_2 \ldots i_l}$ be shorthand for  $\sP^{(k)}_{s_{i_1}} \sP^{(k)}_{s_{i_2}} \cdots \sP^{(k)}_{s_{i_l}}$.  Note that $ \nsH^{(2)}_W = \nsH_W$,  $\sP^{(2)}_s = \sP_s$ and $ \nsH^{(1)}_W = \H_W$,  $\sP^{(1)}_s = \C_s$, and we set  $ \nsH^{(0)}_W = \mathbf{A}$.

\begin{remark}
Proposition \ref{p H2 Hopf algebra} supports the idea that the $\nsH^{(k)}_W$ are  ``the smallest approximation to a Hopf algebra on $\H_W$ deforming the Hopf algebra $\ZZ W$.''  An interesting problem is to make this precise by showing that $\prod_{k \geq 0} \nsH^{(k)}_W$ (or a similar algebra) is a universal object in some categorical sense. This may be closely related to the problem of constructing a right adjoint to the forgetful functor from Hopf algebras to algebras, which was recently done in  \cite{Agore}.
\end{remark}

For any nonnegative integers $k_l, k_r$ with  $k_l + k_r =k$, we have $\Delta^{(k)} = (\Delta^{(k_l)} \tsr \Delta^{(k_r)}) \circ \Delta$; applying this to  $\two^{k} p_{+}$ yields
\renewcommand{\minalignsep}{0pt}
\begin{flalign}
\label{e sP^k definition1} \sP^{(k)}_s &= \sP^{(k_l)}_s \tsr \sP^{(k_r)}_s
  + \sQ^{(k_l)}_s \tsr \sQ^{(k_r)}_s \\
\label{e sP^k definition2} &= 2 \sP^{(k_l)}_s \tsr \sP^{(k_r)}_s - [2]^{k_l} \tsr \sP^{(k_r)}_s - \sP^{(k_l)}_s \tsr [2]^{k_r} + [2]^{k}
\end{flalign}
for all $s \in S$. As a consequence, there is a canonical inclusion
\be
\label{e delta kl kr definition}
\nsbr{\Delta}^{(k_{l}, k_{r})} :  \nsH^{(k)}_{W} \hookrightarrow \nsH^{(k_{l})}_{W} \tsr \nsH^{(k_{r})}_{W}, \
 \sP^{(k)}_s \mapsto \sP^{(k_l)}_s \tsr \sP^{(k_r)}_s
  + \sQ^{(k_l)}_s \tsr \sQ^{(k_r)}_s.
\ee
Thus any pair $\nsbr{M}_{l}, \nsbr{M}_{r},$ where $\nsbr{M}_{l}$ is a $\nsH^{(k_{l})}_{W}$-module and $\nsbr{M}_{r}$ is a $\nsH^{(k_{r})}_{W}$-module, gives rise to the $\nsH^{(k)}_{W}$-module $\nsbr{M}_{l} \tsr \nsbr{M}_{r}$ by restriction.

We have the following commutativity property
\be
\label{e commutativity nsH k}
\Res_{\nsH^{(k)}_{W}} (\nsbr{M}_{l} \tsr \nsbr{M}_{r}) \cong \Res_{\nsH^{(k)}_{W}} (\nsbr{M}_{r} \tsr \nsbr{M}_{l}),
\ee
where the isomorphism is given by the map swapping tensor factors; this is an $ \nsH^{(k)}_W$-module homomorphism by \eqref{e sP^k definition1}.

There are one-dimensional trivial and sign representations of $\nsH^{(k)}_W$ which generalize those for  $\nsH_W = \nsH^{(2)}_W$ defined previously. We also denote these by  $\nsbr{\epsilon}_{+}$ and $\nsbr{\epsilon}_{-}$:
\[ \begin{array}{cccc}
  \nsbr{\epsilon}_{+} : \sP^{(k)}_{s} \mapsto \two^{k}, && \nsbr{\epsilon}_{-} : \sP^{(k)}_s \mapsto 0, & s \in S.
\end{array} \]

For a ring homomorphism $\field \to \mathbf{A}$, we have the specialization $\field \nsH^{(k)}_W := \field \tsr_{\mathbf{A}} \nsH^{(k)}_W$ of the nonstandard Hecke algebra.  Also, we often abuse notation and write $\nsbr{\epsilon}_+$ for the $\field \nsH^{(k)}_W$-module  $\field \tsr_\mathbf{A} \nsbr{\epsilon}_+$, when  $\field$ is understood from context; the same goes for all the other one-dimensional representations in this paper.

\subsection{}
\label{ss theta on nsH k reps}
Here we record some useful results about the anti-automorphisms $1^\text{op}$,  $\theta^\text{op}$ of  $\H_W$ and their corresponding anti-automorphisms $1^\text{op}$,  $(\theta^{(k)})^\text{op}$ of  $\nsH^{(k)}_W$.
Many of the observations here are also made in \cite[\textsection 10]{GCT4}.

Any anti-automorphism $S$ of an $\mathbf{A}$-algebra $H$ allows us to define duals of $H$-modules: let $\langle , \rangle: M \tsr M^* \to \mathbf{A}$ be the canonical pairing, where $M^*$ is the $\mathbf{A}$-module  $\hom_{\mathbf{A}}(M,\mathbf{A})$. Then the $H$-module structure on  $M^*$ is defined by
\[ \langle m,h m'\rangle = \langle S(h)m,m' \rangle \text{ for any } h \in H,\ m \in M, m' \in M^*.\]

We write $M^\dualone$ (respectively  $M^\dualtheta$) for the  $\H_W$-module dual to $M$ corresponding to the anti-automorphism  $1^\text{op}$ (respectively  $\theta^\text{op}$).

We note that for $W = \S_r$ dualization $M \mapsto M^\dualtheta$ corresponds to transposing partitions.
\begin{proposition}[\cite{B4} (see also {\cite[Exercises 2.7, 3.14]{Mathas}})]
\label{p dual basis to C' S = 1 theta op}
Let $M_\lambda$ be the Specht module of  $\H_r$ of shape $\lambda$.
Then
\[
\begin{array}{lcr}
 M_\lambda^\dualone \cong M_{\lambda} &  \text{and} & M_\lambda^\dualtheta \cong M_{\lambda'},
\end{array}
\]
where  $\lambda'$ is the transpose of the partition $\lambda.$
\end{proposition}


Let $\thetaalg$ be the subgroup of automorphisms of $\H_{W}^{\tsr k}$ generated by
\[\theta_i = 1 \tsr \cdots \tsr 1 \tsr \theta \tsr 1 \tsr \cdots \tsr 1,\]
where the $\theta$ appears in the $i$-th tensor factor.
Let $\thetaalg^0$ be the subgroup of $\thetaalg$ generated by $\theta_i \theta_j$ for  $i,j \in [k]$.  This is a subgroup of index $2$ and consists of the involutions having an even number of $\theta$'s and the rest $1$'s.
\begin{proposition}
\label{p theta on nsH k}
The elements of $\thetaalg$ behave well upon restriction to $\aut(\nsH^{(k)}_W)$:
\begin{list}{\emph{(\roman{ctr})}} {\usecounter{ctr} \setlength{\itemsep}{1pt} \setlength{\topsep}{2pt}}
\item $\alpha(\sP_s^{(k)}) = \begin{cases} \sP^{(k)}_{s} & \text{if } \alpha \in \thetaalg^{0} \\
\sQ^{(k)}_{s} & \text{if } \alpha \in \theta_{1} \thetaalg^{0} \end{cases} \text{ for all } s \in S.$
\item $\nsH^{(k)}_W$ is invariant under the action of $\thetaalg$, i.e. $\alpha(\nsH^{(k)}_W) = \nsH^{(k)}_W$ for all $\alpha \in \thetaalg$.
\item  The restriction of an element of $\thetaalg$ to an automorphism of $\nsH^{(k)}_W$ (which is well-defined by (ii)) corresponds to the map $\thetaalg \to \aut (\nsH^{(k)}_W), \theta_i \mapsto \theta^{(k)}$, where $\theta^{(k)}: \nsH^{(k)}_W \to \nsH^{(k)}_W$ is the  $\mathbf{A}$-algebra homomorphism determined by  $\theta^{(k)}(\sP^{(k)}_s) = \sQ^{(k)}_s$ for all $s \in S$.
\item The homomorphism $\theta^{(k)}$ is an involution of  $\nsH^{(k)}_W$.
\item The inclusion  $\nsbr{\Delta}^{(k_l,k_r)}$ of \eqref{e delta kl kr definition} induces a map \newline $\aut( \nsH^{(k_l)}_W \tsr \nsH^{(k_r)}_W) \to \aut( \nsH^{(k)}_W)$, which sends  $\theta^{(k_l)} \tsr 1$ and $1 \tsr \theta^{(k_r)}$ to  $\theta^{(k)}$.
\end{list}
\end{proposition}
\begin{proof}
It is enough to check (i) for  $W = \S_2$, and this can be seen directly from the observation $\theta(p_+) = p_-$ and by comparing
\[ \frac{1}{[2]^k} \sQ^{(k)}_1 = \Delta^{(k)}(p_-) =  \sum_{\substack{ \mathbf{a} \in \{ +, - \}^{k} \\ |\{ i : a_i = - \}| \text{ is odd}} } p_{a_1} \tsr \dots \tsr p_{a_k}\]
with the similar expression for  $\Delta^{(k)}(p_+)$ in \eqref{e explicit coproduct}.  The remaining statements follow easily.
\end{proof}

There are also anti-automorphisms $1^{\text{op}} := 1^{\text{op}} \tsr 1^{\text{op}} \tsr \cdots \tsr 1^{\text{op}}$ and $(\theta^{(k)})^\text{op} := 1^{\text{op}} \circ \theta^{(k)}$ of $\nsH^{(k)}_W$, where $\theta^{(k)}$ is defined in Proposition \ref{p theta on nsH k} (iii). We also write $M^\dualone$ (respectively  $M^\dualtheta$) for the  $\nsH^{(k)}_W$-module dual to $M$ corresponding to the anti-automorphism  $1^\text{op}$ (respectively  $(\theta^{(k)})^\text{op}$).  The next proposition is immediate from Proposition \ref{p theta on nsH k} (v) and definitions.
\begin{proposition}
\label{p 1 op theta op representation facts}
Let $\nsbr{M}_{l}$ be an $\nsH^{(k_{l})}_{W}$-module and $\nsbr{M}_{r}$ an $\nsH^{(k_{r})}_{W}$-module and assume that these are free and finite-dimensional as $\mathbf{A}$-modules.
Then
\[
\begin{array}{ccc}
 (\nsbr{M}_l \tsr \nsbr{M}_r)^\dualone \cong \nsbr{M}_l^\dualone \tsr \nsbr{M}_r^\dualone \cong \nsbr{M}_l^\dualtheta \tsr \nsbr{M}_r^\dualtheta, \\
 (\nsbr{M}_l \tsr \nsbr{M}_r)^\dualtheta \cong \nsbr{M}_l^\dualone \tsr \nsbr{M}_r^\dualtheta \cong \nsbr{M}_l^\dualtheta \tsr \nsbr{M}_r^\dualone.
\end{array}
\]
\end{proposition}

\section{The nonstandard braid relation}
\label{s A generalization of the braid relation}
Here we determine the irreducible representations of the nonstandard Hecke algebra $\nsH^{(k)}_3$ and find that it has a two-dimensional irreducible with defining constant $\two^k T_j(\frac{1}{\two})$ for each $j \in [k]$.  We deduce from this the nonstandard braid relation \eqref{e nonstandard braid relation 0} for $\nsH^{(k)}_3$.

\subsection{}
In this subsection we let $W$ be any dihedral group, whereas in \textsection\ref{ss chebyshev polynomial intro}--\ref{ss nonstandard braid relation H3} we focus on the case  $W = \S_3$.
It will be shown in the course of the next two sections that all of the two-dimensional irreducible representations of $\field \nsH^{(k)}_W$ (for a suitable field $\field$) are of the form $\nsbr{X}^{(k)}(c)\cong \field^2,$ for some constant $c \in \field$, defined by the following matrices giving the action of $\sP^{(k)}_i$ on $\nsbr{X}^{(k)}(c)$:
\be
\label{e bar X action}
\sP^{(k)}_1 \mapsto \mat{[2]^{k} & c \\ 0 & 0},\ \sP^{(k)}_2 \mapsto \mat{0 & 0 \\ 1 & [2]^{k} }.
\ee
Here we have specified a basis $(x_1, x_2)$ for  $\nsbr{X}^{(k)}(c)$ and are thinking of matrices as acting on the left on column vectors, so that the $j$-th column of these matrices gives the coefficients of $\sP^{(k)}_i x_j$ in the basis $(x_1, x_2)$.
The map $\field \nsH^{(k)}_W \to \End_\field(\field^2)$ specified by \eqref{e bar X action} only defines an algebra homomorphism for special values of  $c$ (i.e. $\nsbr{X}^{(k)}(c)$ is only a representation of  $\field \nsH^{(k)}_W$ for special values of  $c$), and these values will be determined in the course of the proofs below.

Recall that a representation of an algebra over a field  $\field$ is \emph{absolutely irreducible} if it is irreducible over any field extension of  $\field$ (the appendix of \cite{Mathas} is a good quick reference for this and other basic definitions and results about finite dimensional algebras over a field).
Suppose that $\nsbr{X}^{(k)}(c)$ is a  $\field \nsH^{(k)}_W$-module.
Observe that $\nsbr{X}^{(k)}(c)$ has a one-dimensional submodule if and only if $\mat{0 \\ 1}$ or $\mat{\two^{k} \\ -1}$ spans a submodule.  Since  $\mat{0 \\ 1}$ spans a submodule if and only if $c = 0$ and $\mat{\two^{k} \\ -1}$ spans a submodule if and only if $c = \f^k$ (recall  $f = [2]^2$), we conclude that
\be
\label{e bar X irreducible}
\text{$\nsbr{X}^{(k)}(c)$ is absolutely irreducible  $\iff \ c \not\in \{0,\f^k\}$.} \\
\ee
One also checks easily that if $\nsbr{X}^{(k)}(c), \nsbr{X}^{(k)}(c')$ are  $\field\nsH^{(k)}_W$-modules, then 
\be
\label{e bar X distinct}
\nsbr{X}^{(k)}(c) \cong \nsbr{X}^{(k)}(c') \iff c=c'.
\ee
\begin{proposition}
\label{p main tensor product computation}
Let  $k_l, k_r$ be positive integers with  $k_l+k_r = k$.  Suppose that $\nsbr{X}_l=\nsbr{X}^{(k_l)}(c_l), \nsbr{X}_r = \nsbr{X}^{(k_r)}(c_r)$ are irreducible $\field \nsH^{(k_l)}_W, \field \nsH^{(k_r)}_W$-modules, respectively with constants $c_l, c_r \in K$. Put  $a_l = \sqrt{c_l},\ a_r = \sqrt{c_r}$ and define
\be
\label{e decompose xl xr constant}
a_\pm = a_l a_r \pm \sqrt{ (\f^{k_l} - c_l) (\f^{k_r} - c_r)}.
\ee
Assume that $a_l a_r, a_\pm \in \field$ and that $\field$ is an integral domain in which $2 \ne 0$.  Then $\nsbr{X}_l \tsr \nsbr{X}_r$ is a  $\field \nsH^{(k)}_W$-module via  $\nsbr{\Delta}^{(k_l,k_r)}$ (see \eqref{e delta kl kr definition}) with the following decomposition into irreducibles
\be
\label{e xl xr decomp}
\nsbr{X}_l \tsr \nsbr{X}_r \cong
\begin{cases}
\nsbr{X}^{(k)}(a_+^2) \oplus \nsbr{X}^{(k)}(a_-^2) & \text{$a_+^2 \neq \f^k$ and  $a_- \neq 0$}, \\
\nsbr{\epsilon}_+ \oplus \nsbr{\epsilon}_- \oplus \nsbr{X}^{(k)}(a_-^2) & \text{$a_+^2 = \f^k$ and  $a_- \neq 0$}, \\
\nsbr{\epsilon}_1 \oplus \nsbr{\epsilon}_2 \oplus \nsbr{X}^{(k)}(a_+^2) &
\text{$a_+^2 \neq \f^k$ and  $a_- = 0$}, \\
\nsbr{\epsilon}_+ \oplus \nsbr{\epsilon}_- \oplus \nsbr{\epsilon}_1 \oplus \nsbr{\epsilon}_2 & \text{$a_+^2 = \f^k$ and  $a_- = 0$},
\end{cases}
\ee
where $\nsbr{\epsilon}_1$ and $\nsbr{\epsilon}_2$ are one-dimensional representations given by
\[
\begin{array}{cll}
\nsbr{\epsilon}_1: & \sP^{(k)}_1 \mapsto \two^k, & \sP^{(k)}_2 \mapsto 0, \\
\nsbr{\epsilon}_2: & \sP^{(k)}_1 \mapsto 0, & \sP^{(k)}_2 \mapsto \two^k.
\end{array}
\]
\end{proposition}
\begin{remark}
A direct calculation shows that, with the appropriate convention for square roots,
\[
\begin{array}{lll}
a_+^2 = \f^k & \iff & \f^{k_r} c_l = \f^{k_l} c_r, \\
a_- = 0 & \iff & \f^{k} = \f^{k_l} c_r + \f^{k_r}c_l.
\end{array}
\]
Note that by \eqref{e bar X irreducible} and \eqref{e decompose xl xr constant}, our assumptions on $\field$  and that $\nsbr{X}_l$ and $\nsbr{X}_r$ are irreducible imply $a_+^2 \neq a_-^2$.  Thus the last three cases of \eqref{e xl xr decomp} are the only degenerate cases that can occur.
\end{remark}
\begin{proof}
Let $(x_{l1}, x_{l2})$ and $(x_{r1}, x_{r2})$ be bases for $\nsbr{X}_l$ and $\nsbr{X}_r$, respectively, corresponding to the matrices in (\ref{e bar X action}).
Using (\ref{e sP^k definition2}) we compute the matrices that give the action of the $\sP^{(k)}_i$ on the $\field \nsH^{(k)}_W$-module  $\nsbr{X}_l \tsr \nsbr{X}_r$ in the basis $(x_{l1} \tsr x_{r1}, x_{l1} \tsr x_{r2}, x_{l2} \tsr x_{r1}, x_{l2} \tsr x_{r2})$:
\be \label{e matrix action on tensor product}
\sP^{(k)}_1 \mapsto \mat{ [2]^k     & [2]^{k_l}c_r  & [2]^{k_r}c_l  & 2c_l c_r \\
        0       & 0             & 0             & -[2]^{k_r} c_l \\
        0       & 0             & 0             & -[2]^{k_l} c_r \\
        0       & 0             & 0             & [2]^k }, \
\sP^{(k)}_2 \mapsto \mat{   [2]^k          & 0          & 0         & 0 \\
        -[2]^{k_l}     & 0          & 0         & 0 \\
        -[2]^{k_r}     & 0          & 0         & 0 \\
        2              & [2]^{k_r}  & [2]^{k_l} & [2]^k }
\ee
Now set
\be
\label{e zi + - def}
z_{1+} = \mat{ a_l a_r a_+ \\ -\two^{k_r} c_l \\ -\two^{k_l} c_r \\ \two^{k} },
z_{2+} = \mat{ [2]^k a_l a_r a_+ \\ -[2]^{k_l} a_l a_r a_+  \\  - \two^{k_r} a_l a_r a_+  \\ a_+^2 },
z_{1-} = \mat{ a_l a_r a_- \\ -\two^{k_r} c_l \\ -\two^{k_l} c_r \\ \two^{k} },
z_{2-} = \mat{ [2]^k a_l a_r a_- \\ -[2]^{k_l} a_l a_r a_-  \\  - [2]^{k_r} a_l a_r a_-  \\ a_-^2}.
\ee


The vectors  $z_{1+}, z_{2+}, z_{1-}, z_{2-}$ were found using the form of the matrices in (\ref{e matrix action on tensor product}) to ensure that $z_{i+}$ and $z_{i-}$ span the $\nsbr{\epsilon}_+$-isotypic component of $\Res_{\{ s_i \}} (\nsbr{X}_l \tsr \nsbr{X}_r)$ for $i = 1,2$.
A direct computation shows that $\field \{z_{1+}, z_{2+}\}$ (respectively  $\field \{z_{1-}, z_{2-}\}$) is a submodule of $\nsbr{X}_l \tsr \nsbr{X}_r$ provided  $a_+$ (respectively  $a_-$) is a solution to the following quadratic equation in the variable  $y$
\be \label{e quadratic a+-}
y^2 - 2 a_l a_r y - \f^{k} + \f^{k_r} c_l + \f^{k_l} c_r = 0.
\ee
The solutions to this quadratic equation are given by (\ref{e decompose xl xr constant}), hence
$\field \{z_{1+}, z_{2+}\}$ and $\field \{z_{1-}, z_{2-}\}$ are submodules.

By comparing the first and last components of $z_{i+}, z_{i-}$, one checks that if $a_+^2 \ne \f^k$ (respectively  $a_- \ne 0$), then $\field \{z_{1+}, z_{2+}\}$ (respectively $\field \{z_{1-}, z_{2-}\}$) is two-dimensional.  Moreover, by the remark before the proof, $a_+^2 \ne \f^k$ (respectively  $a_- \ne 0$) implies $\nsbr{X}^{(k)}(a^2_+)$ (respectively $\nsbr{X}^{(k)}(a^2_-)$) is irreducible.
Then with the assumptions of the first case of \eqref{e xl xr decomp}, $(z_{1+}, z_{2+}, z_{1-}, z_{2-})$ is a basis of $\nsbr{X}_l \tsr \nsbr{X}_r$ and the action of $\sP^{(k)}_i$ in this basis is given by
\be
\sP^{(k)}_1 \mapsto \mat{[2]^k &  a_+^2 & 0 & 0  \\
0 & 0 & 0 & 0 \\
0 & 0 & [2]^k & a_-^2 \\
0 & 0 & 0 & 0 }, \
\sP^{(k)}_2 \mapsto \mat{0 & 0 & 0 & 0 \\
1 & [2]^k & 0 & 0 \\
0 & 0 & 0 & 0 \\
0 & 0 & 1 & [2]^k },
\ee
which verifies (\ref{e xl xr decomp}) in this case.

Next suppose $a_+^2 = \f^k$ and $a_- \neq 0$, and define $z_{2+}'$ by
\[
z_{2+}' := \mat{ 0 \\ 1 \\  -1  \\ 0 }.
\]
Then $(z_{1+}, z_{2+}', z_{1-}, z_{2-})$ is a basis and the action of $\sP^{(k)}_i$ in this basis is given by
\be
\sP^{(k)}_1 \mapsto \mat{[2]^k & 0 & 0 & 0  \\
0 & 0 & 0 & 0 \\
0 & 0 & [2]^k & a_-^2 \\
0 & 0 & 0 & 0 }, \
\sP^{(k)}_2 \mapsto \mat{\two^k & 0 & 0 & 0 \\
0 & 0 & 0 & 0 \\
0 & 0 & 0 & 0 \\
0 & 0 & 1 & [2]^k }.
\ee
The second case of (\ref{e xl xr decomp}) follows.
The third and fourth cases of \eqref{e xl xr decomp} are similar: if $a_+^2 \ne \f^k$ and  $a_- = 0$, then $(z_{1+}, z_{2+}, z_{1-}, z_{2-}')$ is a basis, where
\[
z_{2-}'  := \mat{ [2]^k a_l a_r  \\ -[2]^{k_l} a_l a_r   \\  - [2]^{k_r} a_l a_r \\ a_-}.
\]
If $a_+^2 = \f^k$ and  $a_- = 0$, then $(z_{1+}, z_{2+}', z_{1-}, z_{2-}')$ is a basis of $\nsbr{X}_l \tsr \nsbr{X}_r$.
%
\end{proof}

\subsection{}
\label{ss chebyshev polynomial intro}
The  $k$-th Chebyshev polynomial $T_k(x)$ of the first kind is the polynomial expressing $\cos(k \theta)$ in terms of $x = \cos(\theta)$. Chebyshev polynomials appear in many areas of mathematics including numerical analysis, special functions, approximation theory, and ergodic theory. Explicit formulas, recurrences, and generating functions are known for Chebyshev polynomials \cite{Rivlin}, though in this paper all we need are simple trigonometric identities.
Recall from the introduction the constants
\be a^{(k)} = T_k \Big( \frac{1}{[2]} \Big),\ k \geq 1.\ee
We will see that $\two^k a^{(j)}$, $j \in [k]$ are the defining constants of the two-dimensional irreducible representations of $\nsH^{(k)}_3$.  The first few coefficients  $\two^k a^{(k)}$ are
\be
\begin{array}{llrll}
\two a^{(1)} &=& 1 & = &\hphantom{\u^4 - 4\u^2 -}1\\
\two^2 a^{(2)} &=& - \f+2 &=& \hphantom{\u^4 3} -\u^2 \hphantom{-2 3} - \u^{-2} \\
\two^3 a^{(3)} &=& - 3\f+4 &=& \hphantom{\u^4} -3\u^2 - 2 - 3\u^{-2} \\
\two^4 a^{(4)} &=& \f^2-8\f+8 &=& \u^4 - 4\u^2 - 2 - 4\u^{-2} + \u^{-4}.
\end{array}
\ee

Though Chebyshev polynomials are usually defined for $k \geq 0$, it is convenient to define them for all integers $k$.  Note that the definition above still makes sense and we have  $T_k(x) = T_{-k}(x)$.  Accordingly, we extend the definition $a^{(k)} = T_k (\frac{1}{[2]})$ to all  $k \in \ZZ$.  Also let  $T^1_k(x)$ be the element of $\sqrt{1-x^2} \ZZ[x]$ obtained by expressing $\sin(k x)$ in terms of  $x = \cos(\theta)$; more precisely, we should write $T^1_k(x) \in y \ZZ[x] \subseteq \ZZ[x,y]/(x^2 + y^2 - 1)$.  The $k$-th Chebyshev polynomial $U_k(x)$ of the second kind is the polynomial expressing $\sin((k+ 1) \theta)/\sin(\theta)$ in terms of $x = \cos(\theta)$. Then we have $T^1_k(x) = \sqrt{1-x^2} U_{k-1}(x)$ for  $k \geq 1$.  Again, we may allow $k$ to be any integer  and there holds $T^1_{-k}(x) = -T^1_{k}(x)$.

The calculation decomposing the tensor products $\nsbr{X}_l \tsr \nsbr{X}_r$ (Proposition \ref{p main tensor product computation}) and the following identity for Chebyshev polynomials are all we need to determine the irreducibles of $\nsH^{(k)}_3$.

\begin{lemma}
\label{l Chebyshev quadratic identity}
For non-negative integers $k_l, k_r$, there holds the following identity for Chebyshev polynomials (omitting the dependence of $T_k(x)$ on $x$)
\[ T_{k_l\mp k_r} = T_{k_l} T_{k_r} \pm T^1_{k_l} T^1_{k_r}.\]
Hence
\[ [2]^ka^{(k_l\mp k_r)} = [2]^{k_l} a^{(k_l)}\, [2]^{k_r} a^{(k_r)} \pm \sqrt{\big(f^{k_l}-([2]^{k_l}a^{(k_l)})^2\big) \big(f^{k_r}-([2]^{k_r}a^{(k_r)})^2\big) }.\]
\end{lemma}
\begin{proof}
The first statement is immediate from the trigonometric identity for the cosine of a sum of angles:
\begin{align*}
  T_{k_l} T_{k_r} \pm T^1_{k_l} T^1_{k_r} = \cos(k_{l} \theta) \cos(k_{r} \theta ) \pm \sin(k_{l} \theta) \sin(k_{r} \theta )
= \cos((k_{l} \mp k_{r})\theta) = T_{k_l \mp k_r}.
\end{align*}
The second statement then follows by setting  $x = \frac{1}{[2]}$, multiplying both sides by  $[2]^k$ and using  $T_k^1 = \sqrt{1-T_k^2}$.
\end{proof}

\begin{theorem}
\label{t nsH S3 2D reps}
Define $\nsbr{M}^{(k)} = \nsbr{X}^{(k)}(([2]^k a^{(k)})^2)$.  For $\field = \QQ(\u),$ the irreducible representations of $\field \nsH^{(k)}_3$ consist of the trivial and sign representations
\[ \nsbr{\epsilon}_+, \ \nsbr{\epsilon}_-, \]
and the $k$ two-dimensional representations
  \[\epsilon_+^{\tsr k - 1} \tsr \nsbr{M}^{(1)}, \epsilon_+^{\tsr k-2} \tsr \nsbr{M}^{(2)}, \dots, \nsbr{M}^{(k)}. \]
\end{theorem}
\begin{proof}
The proof is by induction on $k$. It is well known that  $\epsilon_+ = \field M_{(3)}, \epsilon_- = \field M_{(1,1,1)},$ and $\field M_{(2,1)}$ are the irreducible representations of $\field \H_3 = \field \nsH^{(1)}_3$, where  $M_\lambda$ is the Specht module of shape  $\lambda$.  A basis of  $\field M_{(2,1)}$ is given by the Kazhdan-Lusztig basis:  the subquotient $\field \{\C_{s_1}, \C_{s_2s_1}, \C_{s_1s_2s_1}\}/\field \{\C_{s_1s_2s_1}\}$ of $\field \H_3$ (considered as a left $\field \H_3$-module) is equal to  $\field M_{(2,1)}$.  In the natural basis $(\C_{s_1}, \C_{s_2s_1})$ of this module, the
matrices for left multiplying by $\sP^{(1)}_i = \C_{s_i}$ are the same as those defining $\nsbr{M}^{(1)} = \nsbr{X}^{(1)}(1)$; the Kazhdan-Lusztig basis elements above are given explicitly by
$\C_{s_2s_1} = \C_{s_2}\C_{s_1},\ \C_{s_1s_2s_1}= \C_{s_1}\C_{s_2}\C_{s_1} - \C_{s_1}$.
This verifies the result for  $k=1$. Now assume $k >1$.

Since $\field \nsH^{(k)}_3$ is a subalgebra of $\field \H_3 \tsr \field \nsH^{(k-1)}_3$, every irreducible representation of $\field \nsH^{(k)}_3$ belongs to the composition series of $X_l \tsr \nsbr{X}_r $ for some irreducible $\field \H_3$-module $X_l$ and irreducible $\field \nsH^{(k-1)}_3$-module $\nsbr{X}_r$. From the case  $X_l = \epsilon_+$, we conclude by induction that $\nsbr{\epsilon}_+,  \nsbr{\epsilon}_-$, and $\epsilon_+^{\tsr j} \tsr \nsbr{M}^{(k-j)}$, $j \in [k-1]$ are distinct irreducibles of $\field \nsH^{(k)}_3$.  To obtain the complete list of irreducibles, by commutativity \eqref{e commutativity nsH k}, it remains to decompose $\nsbr{M}^{(1)} \tsr \nsbr{M}^{(k-1)}$ and $\nsbr{\epsilon}_- \tsr \nsbr{M}^{(k-1)}$ into irreducibles.

Decomposing $\nsbr{M}^{(1)} \tsr \nsbr{M}^{(k-1)}$ into irreducibles is the crux of the proof, and this is a special case of Proposition \ref{p main tensor product computation}.  Note that $\epsilon_+^{\tsr 2} \tsr \nsbr{M}^{(k-2)} = \nsbr{X}^{(k)}((\two^k a^{(k-2)})^2).$  Then by Proposition \ref{p main tensor product computation} and Lemma \ref{l Chebyshev quadratic identity} with  $k_l = 1, k_r = k-1$, $\nsbr{M}^{(1)} \tsr \nsbr{M}^{(k-1)}$ decomposes into irreducibles as
\begin{alignat}{5} \label{e m1 mk-1 tensor decomposition}
\nsbr{M}^{(1)} & \tsr \nsbr{M}^{(k-1)} & \ \cong \ & \left( \epsilon_+^{\tsr 2} \tsr \nsbr{M}^{(k-2)} \right) & \oplus \nsbr{M}^{(k)} & \text{\qquad if $k>2$}, \\
\label{e m1 m1 tensor decomposition}
\nsbr{M}^{(1)} & \tsr \nsbr{M}^{(1)} &\ \cong \ & \qquad \nsbr{\epsilon}_+ \oplus \nsbr{\epsilon}_-  & \oplus \nsbr{M}^{(2)} & \text{\qquad if $k = 2$}.
\end{alignat}


Finally, $\epsilon_- \tsr \nsbr{M}^{(k-1)} \cong \epsilon_-^\dualone \tsr (\nsbr{M}^{(k-1)})^\dualone \cong \epsilon_+ \tsr (\nsbr{M}^{(k-1)})^{\dualtheta}$, where the second isomorphism is by Proposition \ref{p 1 op theta op representation facts} and the first is by $\nsbr{M}^{(k-1)} \cong (\nsbr{M}^{(k-1)})^{\dualone} \cong (\nsbr{M}^{(k-1)})^\dualtheta$.  This last fact follows by induction using (\ref{e m1 mk-1 tensor decomposition}) and (\ref{e m1 m1 tensor decomposition}); the base case $\nsbr{M}^{(1)} \cong (\nsbr{M}^{(1)})^{\dualone} \cong (\nsbr{M}^{(1)})^\dualtheta$ is the $r=3$ case of Proposition \ref{p dual basis to C' S = 1 theta op}.  Thus the irreducible constituent of $\epsilon_- \tsr \nsbr{M}^{(k-1)} $ is already in our list.


It remains to check that $\nsbr{M}^{(k)}$ is distinct from $\epsilon_+^{\tsr j} \tsr \nsbr{M}^{(k-j)}$,  $j \in [k-1]$. This holds by \eqref{e bar X distinct} and the fact that $T_k(x)$ is a degree  $k$ polynomial in  $x$, implying  $[2]^k a^{(j)}$ is a polynomial in the  $\u$-integer $[2]$ whose constant coefficient is nonzero if and only if $j =k$.
\end{proof}

\begin{corollary}\label{c nsh 3 semisimple}
The algebra $\field \nsH^{(k)}_3$ is split semisimple for $\field = \QQ(u)$.
\end{corollary}
\begin{proof}
Proceed by induction on $k$.  We can view the inclusion $\nsbr{\Delta}^{(1,k-1)}: \nsH^{(k)}_3 \to \H_3 \tsr \nsH^{(k- 1)}_3$ of \eqref{e delta kl kr definition} as one of left $\nsH^{(k)}_3$-modules.  Since localizations are flat, $\field$ is a flat $\mathbf{A}$-module; thus, $\field \nsH^{(k)}_3 \to \field \H_3 \tsr \field \nsH^{(k- 1)}_3$ is also an inclusion (this fails for the specialization  $\mathbf{A} \to \ZZ$,  $\u \mapsto 1$). Hence to show that  $\field \nsH^{(k)}_3$ is semisimple it suffices to show that $X_l \tsr \nsbr{X}_r$ is a direct sum of irreducible $\field \nsH^{(k)}_3$-modules for any irreducible $\field \H_3$-module $X_l$ and irreducible $\field \nsH^{(k-1)}_3$-module $\nsbr{X}_r$.  The proof above gives an explicit decomposition of $X_l \tsr \nsbr{X}_r$ into irreducibles.

The algebra $\field \nsH^{(k)}_3$ is split because the irreducibles $\epsilon^{\tsr k-j} \tsr \nsbr{M}^{(j)} \cong \nsbr{X}^{(k)}(([2]^k a^{(j)})^2)$,  $j \in [k]$ are absolutely irreducible by \eqref{e bar X irreducible}.
\end{proof}

\begin{remark}
\label{r semisimple}
It was claimed in \cite{GCT4} that $\RR \nsH_r$ is semisimple for any specialization  $\mathbf{A} \to \RR$, however there is a mistake in the proof. In fact, the specialization $\u \mapsto 1$ is not semisimple if $r > 2$.
The proof was later repaired in  \cite[Proposition 11.8]{BMSGCT4} to show that  $\QQ(\u) \nsH_r$ is semisimple.  Since  $\QQ(u)$ is a perfect field, it follows that $\QQ(u) \nsH_r$ is a separable algebra, which means that  $\field \nsH_r$ is semisimple for any field extension  $\field \supseteq \QQ(u)$.  So, for instance,  $\RR(u) \nsH_r$ is also semisimple.
We strongly suspect that all the algebras $\QQ(u) \nsH^{(k)}_W$ are semisimple for  $W$ any finite Weyl group, but we do not yet have a proof.
\end{remark}

\subsection{}
\label{ss nonstandard braid relation H3}
Now we can determine the nonstandard braid relation for $ \nsH^{(k)}_3$.
Define
\[ F_k(y) = (y - (\two^{k} a^{(1)})^2) (y - (\two^{k} a^{(2)})^2) \dots (y - (\two^k a^{(k)})^2), \]
a polynomial in an indeterminate $y$ with coefficients in $\mathbf{A}$.


\begin{theorem}
\label{t S3 coprod k relations}
The algebra $\nsH^{(k)}_3$ is the associative $\mathbf{A}$-algebra generated by $\sP^{(k)}_s, s \in S=\{s_1,s_2\},$ with quadratic relations
\be
\label{e nsH k quadratic relation}
(\sP^{(k)}_s)^2 = [2]^k \sP^{(k)}_s, \ \ \ s \in S,
\ee
and nonstandard braid relation
\be
\label{e nsH k braid relation}
\sP^{(k)}_1 F_k(\sP^{(k)}_{21}) = \sP^{(k)}_2 F_k(\sP^{(k)}_{12}).
\ee
\end{theorem}

\begin{proof}
The quadratic relations follow from the fact that $\H_2$ is a Hopf algebra.

Set
\[h:=\sP^{(k)}_1 F_k(\sP^{(k)}_{21}) - \sP^{(k)}_2 F_k(\sP^{(k)}_{12}) \in \nsH^{(k)}_3.\]
We next show that the nonstandard braid relation holds in $\field \nsH^{(k)}_3$, for $\field = \QQ(\u),$ i.e. $1 \tsr h = 0$ in $\field \nsH^{(k)}_3$.  To see this, one computes easily using \eqref{e bar X action} that $\sP^{(k)}_2 (\sP^{(k)}_{12} - (\two^k a^{(j)})^2)$ and $\sP^{(k)}_1 (\sP^{(k)}_{21} - (\two^k a^{(j)})^2)$ act on $\epsilon_+^{\tsr k-j} \tsr \nsbr{M}^{(j)}$ by $0$.
We also have that $\sP^{(k)}_i$ acts by $0$ on $\nsbr{\epsilon}_-$.  Further noting that both sides of (\ref{e nsH k braid relation}) act on $\nsbr{\epsilon}_+$ by the constant $[2]^k F_k([2]^{2k})$, we conclude using Theorem \ref{t nsH S3 2D reps} that $1 \tsr h$ acts by $0$ on all irreducible representations of $\field \nsH^{(k)}_3.$  The semisimplicity of $\field \nsH^{(k)}_3$ then implies $1 \tsr h=0$.

We next claim that $\nsH^{(k)}_3 \to \field \nsH^{(k)}_3$ is injective, which would imply  $h=0$, i.e. the nonstandard braid relation holds.  The claim holds because  $\field \nsH^{(k)}_3$ is the localization of $ \nsH^{(k)}_3$ at the multiplicative set  $U = \mathbf{A} \backslash \{0\}$ and  $U$ contains no zero divisors (the facts about localization needed for this would be standard in the commutative setting, and it is not difficult to show that they carry over to the case that the multiplicative set lies in the center of the ring).  The multiplicative set $U$ contains no zero divisors because $\nsH^{(k)}_3$ is a subalgebra of the free $\mathbf{A}$-module  $\H_W^{\tsr k}$.

To see that no other relations hold in $ \nsH^{(k)}_3$, let $H$ be the $\mathbf{A}$-algebra generated by the $\sP^{(k)}_i$ with relations \eqref{e nsH k quadratic relation} and \eqref{e nsH k braid relation}.  The algebra $H$ is a free  $\mathbf{A}$-module with basis given by the  $4k+2$ monomials
 \[ 1, \ (\sP^{(k)}_{12})^j,\ (\sP^{(k)}_{21})^j, \ \sP^{(k)}_1 (\sP^{(k)}_{21})^{j-1}, \ \sP^{(k)}_2 (\sP^{(k)}_{12})^{j-1}, \sP^{(k)}_1 (\sP^{(k)}_{21})^k, \quad j \in [k]. \]
By the split semisimplicity of $\QQ(u)\nsH^{(k)}_3$, $\dim_{\QQ(u)} \QQ(u) \nsH^{(k)}_3 = 4k+2=  \dim_{\QQ(u)} \QQ(u)H$. An additional relation holding in $\nsH^{(k)}_3$ and not in $H$ would force $\dim_{\QQ(u)} \QQ(u) \nsH^{(k)}_3 <  \dim_{\QQ(u)} \QQ(u)H$, hence $H \cong \nsH^{(k)}_3$.
\end{proof}

In the proof above we have shown that
\begin{corollary}
The algebra  $\nsH^{(k)}_3$ is free as an  $\mathbf{A}$-module.
\end{corollary}

\section{Generalizations to dihedral groups}
\label{s Generalizations to dihedral groups}

Let $W$ be the dihedral group of order $2m$ with simple reflections $S= \{ s_1, s_2 \}$ and relations
\be
\begin{array}{cc}
s^2=\idelm, & s \in S \\
(s_1 s_2)^m = \idelm. &
\end{array}
\ee
We will generalize the results of the previous section to the algebras $\nsH^{(k)}_W$.  The two-dimensional irreducibles of $ \nsH^{(k)}_W$ are indexed by signed compositions of  $k$ (Definition \ref{d signed compositions})
and their defining constants are given by the evaluation of a multivariate generalization of Chebyshev polynomials.  These multivariate
Chebyshev polynomials seem to be new. For instance, they do not agree with those studied in \cite{Beerends}.

\subsection{}
We first recall the description of the irreducible representations of $\H_W$.

\begin{theorem}[Kilmoyer-Solomon (see, for instance, {\cite[Theorem 8.3.1]{GeckP}})]
\label{t dihedral irreducibles}
Let $\zeta$ be a primitive $m$-th root of unity and define $w_j=2+\zeta^j+\zeta^{-j}$. Set $n=\frac{m-2}{2}$ if $m$ is even and $n=\frac{m-1}{2}$ if $m$ is odd.  Suppose $\field$ is a subfield of $\CC(u)$ containing the $w_j$.  The irreducible representations of $\field \H_W$ consist of
\begin{list}{\emph{(\alph{ctr})}} {\usecounter{ctr} \setlength{\itemsep}{1pt} \setlength{\topsep}{2pt}}
\item the trivial and sign representations $\epsilon_+$, $\epsilon_-$,
\item $n$ two-dimensional representations denoted $M_j, j \in [n]$, where $M_j=\nsbr{X}^{(1)}(w_j)$,
\item if $m$ is even, two one-dimensional representations $\epsilon_1$ and $\epsilon_2$ determined by
\begin{alignat*}{2}
\epsilon_1:\quad  \C_1 &\mapsto \two, & \qquad \C_2 &\mapsto 0, \\
\epsilon_2:\quad  \C_1 &\mapsto 0, & \qquad \C_2 &\mapsto \two.
\end{alignat*}
\end{list}
\end{theorem}

Throughout this section, $W$ denotes the dihedral group of order  $2m$ and we maintain the notation  $n, w_j, M_j$, etc. of Theorem \ref{t dihedral irreducibles}.

\subsection{}
Here we introduce the combinatorics and multivariate generalization of Chebyshev polynomials needed to describe the irreducibles of $ \nsH^{(k)}_W$.

\begin{definition}
\label{d signed compositions}
A \emph{signed  $n$-composition of $k$} is an  $n$-vector  $\mathbf{k} = (k_1,\ldots,k_n)$ of integers such that $\sum_{i =1}^n |k_i| = k$.  If $\mathbf{k}$ is a signed composition of  $k$, we also write $|\mathbf{k}| = k$.  Two signed $n$-compositions $\mathbf{k, k'}$ are \emph{equivalent} if $\mathbf{k=k'}$ or $\mathbf{k=-k'}$.

Let $\compositions{n}{k}$ be the set of equivalence classes of signed $n$-compositions of $k$. Write  $\compositions{n}{\leq k} = \bigcup_{1 \leq k' \leq k} \compositions{n}{k'}$ and $\compositions{n}{< k} = \bigcup_{1 \leq k' < k} \compositions{n}{k'}$.
\end{definition}

\begin{example}
The signed  $2$-compositions of $3$, with equivalent compositions in the same column, are
\[
\begin{array}{cccccc}
(3,0) &(2, 1) & ( 2, \ng 1) & (1, 2) & (1, \ng 2) & (0, 3)  \\
(\ng 3,0) &(\ng 2, \ng 1) & ( \ng 2, 1) & (\ng 1, \ng 2) & (\ng 1, 2) & (0, \ng 3).
\end{array}
\]
Since the number of $i$-compositions of $k$ is $\binom{k-1}{i-1}$, we have the enumerative result
\[ | \compositions{n}{k}|=\sum_{i=1}^{n} 2^{i-1} \binom{n}{i} \binom{k-1}{i-1}. \]
\end{example}

\begin{definition}
Let $\mathbf{k}$ be a signed $n$-composition of  $k$.  The \emph{multivariate Chebyshev polynomial} $T_{\mathbf{k}}(x_1,\dots,x_n, y_1,\ldots,y_n) \in R$ is the polynomial expressing \newline
$\cos(k_1 \theta_1 + k_2 \theta_2 + \dots + k_n \theta_n)$ in terms of $x_j = \cos(\theta_j), y_j = \sin(\theta_j)$, where
\[R = \ZZ[x_1,\ldots,x_n,y_1,\ldots,y_n]/\bigoplus_j (x_j^2+y_j^2 - 1). \]

Similarly, $T^1_{\mathbf{k}}(x_1,\dots,x_n, y_1,\ldots,y_n) \in R$ is the polynomial expressing \newline $\sin(k_1 \theta_1 + k_2 \theta_2 + \dots + k_n \theta_n)$ in terms of $x_j = \cos(\theta_j), y_j = \sin(\theta_j)$.
\end{definition}

The polynomials  $T_\mathbf{k}, T^1_\mathbf{k}$ can be expressed in terms of the univariate polynomials $T^0_k(x) := T_k(x, y)$ and $T^1_k(x,y) = y U_{k-1}(x)$ (as defined in \textsection\ref{ss chebyshev polynomial intro}) explicitly as follows. Let $\FF_2$ be the finite field of order $2$.
If $\mathbf{k}$ is an $n$-vector, then a vector $\bm{\alpha} \in \FF_2^n$ is \emph{$\mathbf{k}$-supported} if $k_j=0$ implies $\alpha_j=0$.
There holds
\begin{align}
 T_{\mathbf{k}}(x_1, \dots, x_n,y_1,\ldots,y_n) &= \displaystyle \sum_{\substack{\bm{\alpha}\ \mathbf{k}\text{-supported},  \\ |\bm{\alpha}| \text{ even}}}\ \prod_{j \in [n]} T^{\alpha_j}_{k_j}(x_j,y_j), \\[3px]
 T^1_{\mathbf{k}}(x_1, \dots, x_n,y_1,\ldots, y_n) &= \displaystyle \sum_{\substack{\bm{\alpha} \ \mathbf{k}\text{-supported},  \\ |\bm{\alpha}| \text{ odd}}}\ \prod_{j \in [n]} T^{\alpha_j}_{k_j}(x_j,y_j).
\end{align}

Fix once and for all square roots $\sqrt{w_j}$ and $\sqrt{\f-w_j}$ in some field $\field$ containing $\mathbf{A}$.  For instance, we may choose $\sqrt{w_{j}} = 2 \cos (\frac{j \pi}{m})$ if $\field$ contains $\RR$.
\begin{definition}
\label{d constants ak}
Define $\mathbf{A^*} = \mathbf{A}[\sqrt{w_1},\ldots,\sqrt{w_n},\sqrt{\f-w_1},\ldots,\sqrt{\f-w_n}]$.
For each signed $n$-composition  $\mathbf{k}$ of  $k$, define the following constants, which after being multiplied by $\two^k$, belong to $\mathbf{A^*}$:
\begin{itemize}{\setlength{\itemsep}{4pt} \setlength{\topsep}{2pt}}
\item $a_{\mathbf{k}} = T_{\mathbf{k}}\left( \frac{\sqrt{w_1}}{\two}, \dots, \frac{\sqrt{w_n}}{\two}, \frac{1}{\two} \sqrt{\f - w_1}, \dots, \frac{1}{\two} \sqrt{\f - w_n} \right)$,
\item $a^1_{\mathbf{k}} = T^1_{\mathbf{k}}\left( \frac{\sqrt{w_1}}{\two}, \dots, \frac{\sqrt{w_n}}{\two}, \frac{1}{\two} \sqrt{\f - w_1}, \dots, \frac{1}{\two} \sqrt{\f - w_n} \right)$.
\end{itemize}
\end{definition}
To better understand these coefficients, it is helpful to compute the specializations
\be
\label{e a coefficients u is 1}
( a_{\mathbf{k}} ) |_{\u=1} = \cos \Big( \frac{\pi}{m} \sum_{j=1}^{n} j k_{j} \Big), \quad  ( a^{1}_{\mathbf{k}} ) |_{\u=1} = \sin \Big( \frac{\pi}{m} \sum_{j=1}^{n} j k_{j} \Big).
\ee


\subsection{}
Let $\field$ be the quotient field of $\mathbf{A^*}$.
We give names to the two-dimensional irreducibles of $\field \nsH^{(k)}_W$:
\be
\begin{array}{lll}
\nsbr{N}^{(k)}_{\mathbf{k'}} &:= \nsbr{X}^{(k)}( (\two^k a_{\mathbf{k'}})^2)& \mathbf{k'} \in \compositions{n}{\leq k}, \\
\nsbr{N}^{1 (k)}_\mathbf{k'} &:= \nsbr{X}^{(k)}( (\two^k a^1_{\mathbf{k'}})^2) \cong \epsilon_1 \tsr \nsbr{N}_{\mathbf{k'}}^{(k-1)}& \mathbf{k'} \in \compositions{n}{ < k}.
\end{array}
\ee

Note that these definitions make sense because  $(\two^k a_{\mathbf{k'}})^2$ and  $(\two^k a^1_{\mathbf{k'}})^2$ are independent of the equivalence class of $\mathbf{k'}.$
The isomorphism in the second line will be seen in the course of the proof below.
We also remark that $\nsbr{N}_{\mathbf{k'}}^{(k)} \cong \epsilon_{+}^{\tsr k - k'} \tsr \nsbr{N}_{\mathbf{k'}}^{(k')}$, where  $k' = |\mathbf{k'}|$, is an isomorphism of  $\field \nsH^{(k)}_W$-modules.  In the special case that $\mathbf{k}$ has only one nonzero component $k_j = k$,  $\nsbr{N}_{\mathbf{k}}^{(k)} \cong \nsbr{X}^{(k)}( (\two^k T_{k}(\frac{\sqrt{w_j}}{[2]}) )^2)$ is similar to the representation $\nsbr{M}^{(k)}$ in the $W = \S_3$ case.

\begin{theorem}
\label{t nsH W 2d reps}
Let $W, m,$ and $n$ be as in Theorem \ref{t dihedral irreducibles}.
For  $\field$ the quotient field of $\mathbf{A^*}$, the irreducible representations of $\field \nsH^{(k)}_W$ consist of
\begin{list}{\emph{(\alph{ctr})}} {\usecounter{ctr} \setlength{\itemsep}{1pt} \setlength{\topsep}{2pt}}
\item the trivial and sign representations $\epsilon_+$, $\epsilon_-$,
\item for each $\mathbf{k'} \in \compositions{n}{\leq k}$, the two-dimensional representation  $\nsbr{N}^{(k)}_\mathbf{k'}$,
\item if $m$ is even, two one-dimensional representations $\nsbr{\epsilon}_1 \cong \epsilon_1 \tsr \epsilon_+^{\tsr k-1}$ and $\nsbr{\epsilon}_2 \cong \epsilon_2 \tsr \epsilon_+^{\tsr k-1}$ given by
\[
\begin{array}{cll}
\nsbr{\epsilon}_1: & \sP^{(k)}_1 \mapsto \two^k, & \sP^{(k)}_2 \mapsto 0, \\
\nsbr{\epsilon}_2: & \sP^{(k)}_1 \mapsto 0, & \sP^{(k)}_2 \mapsto \two^k,
\end{array}
\]
\item if $m$ is even, for each $\mathbf{k'} \in \compositions{n}{< k}$, the two-dimensional representation $\nsbr{N}^{1 (k)}_\mathbf{k'}$.
\end{list}
\end{theorem}

\begin{proof}
The proof is by induction on $k$ and similar to that of Theorem \ref{t nsH S3 2D reps}.  The base case $k=1$ is Theorem \ref{t dihedral irreducibles}.

As in the proof of Theorem \ref{t nsH S3 2D reps}, it suffices to decompose $X_{l}\tsr \nsbr{X}_{r}$ into irreducibles for every irreducible $\field \H_W$-module $X_l$ and irreducible $\field \nsH^{(k-1)}_W$-module $\nsbr{X}_r$.
The calculation we need is a special case of the following:
suppose that $k_l, k_r$ are positive integers such that  $k= k_l+k_r$ and $\mathbf{k}_l,\mathbf{k}_r$ are signed $n$-compositions of $k_l$, $k_r$ respectively.
Then by Proposition \ref{p main tensor product computation} and Lemma \ref{l Chebyshev a mathbf k} (below), the $\field \nsH^{(k)}_W$-module $\nsbr{N}^{(k_{l})}_{\mathbf{k}_l} \tsr \nsbr{N}^{(k_{r})}_{\mathbf{k}_r}$ decomposes into irreducibles as (assume without loss of generality that if $\mathbf{k}_l$ is equivalent to $\mathbf{k}_r$, then $\mathbf{k}_l = \mathbf{k}_r$)
\begin{alignat}{2}
\label{e general tensor product rule N bfk}
\nsbr{N}^{(k_{l})}_{\mathbf{k}_l} \tsr \nsbr{N}^{(k_{r})}_{\mathbf{k}_r} &\cong \nsbr{N}^{(k)}_{\mathbf{k}_l - \mathbf{k}_r} \oplus \nsbr{N}^{(k)}_{\mathbf{k}_l + \mathbf{k}_r}
&\text{\qquad if $\mathbf{k}_l \neq \mathbf{k}_r$}, \\
\label{e general tensor product rule N bfk 2}
\nsbr{N}^{(k_{l})}_{\mathbf{k}_l} \tsr \nsbr{N}^{(k_{r})}_{\mathbf{k}_r} &\cong \nsbr{\epsilon}_+ \oplus \nsbr{\epsilon}_- \oplus \nsbr{N}^{(k)}_{\mathbf{k}_l + \mathbf{k}_r}  &\text{\qquad if $\mathbf{k}_l = \mathbf{k}_r$.}
\end{alignat}
In particular, if we are given some signed $n$-composition $\mathbf{k}$ of $k$, then we may certainly choose $\mathbf{k}_l,\mathbf{k}_r$ such that  $\mathbf{k}_l + \mathbf{k}_r = \mathbf{k}$. Thus the first summand of the right-hand side of (\ref{e general tensor product rule N bfk}) is  $\nsbr{N}^{(k)}_\mathbf{k}$, so this is a representation of  $\field \nsH^{(k)}_W$.  And certainly $|\mathbf{k}_{l}-\mathbf{k}_{r}|$, $|\mathbf{k}_{l}+\mathbf{k}_{r}| \leq k$, so the irreducibles on the right-hand side of \eqref{e general tensor product rule N bfk} and \eqref{e general tensor product rule N bfk 2} are in our list.

We now must consider the cases $X_l = \epsilon_+, \epsilon_-, \epsilon_1, \epsilon_2$.
For the case $X_l = \epsilon_+$, we have $\epsilon_{+} \tsr \nsbr{N}_{\mathbf{k'}}^{(k-1)} \cong \nsbr{N}_{\mathbf{k'}}^{(k)}$ and
$\epsilon_{+} \tsr \nsbr{N}_{\mathbf{k'}}^{1 (k-1)} \cong \nsbr{N}_{\mathbf{k'}}^{1 (k)}$, and these are in our list of irreducibles (where $\mathbf{k'} \in \compositions{n}{ \le k-1}$,  $\mathbf{k'} \in \compositions{n}{ < k-1}$ for the first, second isomorphism respectively).

Next, we consider the case $X_l = \epsilon_-$. We have
\be \label{e dual bar N}
\epsilon_- \tsr \nsbr{X}_r \cong \epsilon_- \tsr ((\nsbr{X}_r)^{\dualone})^\dualone \cong \epsilon_+ \tsr ((\nsbr{X}_r)^\dualone)^{\dualtheta},
\ee
where the second isomorphism is by Proposition \ref{p 1 op theta op representation facts}.  Hence this case follows from the  $X_l = \epsilon_+$ case.

Finally, we consider the case that $m$ is even and $X_{l}=\epsilon_{1}$ or $X_{l}=\epsilon_{2}$. The action of $\sP^{(k)}_i$ on $\epsilon_1 \tsr \nsbr{X}^{(k-1)}(c)$ is given by
\be
\label{e E bar X action}
\sP^{(k)}_1 \mapsto \mat{[2]^{k} &  \two c \\ 0 & 0},\ \sP^{(k)}_2 \mapsto \mat{\two^k & 0 \\ -\two & 0 }.
\ee
Changing to the basis
\[ \left( \mat{1 \\ 0 } , \mat{\two^k \\ \two} \right) \]
shows that $\epsilon_1 \tsr \nsbr{X}^{(k-1)}(c) \cong \nsbr{X}^{(k)}(\f^k - c \f)$, hence $\epsilon_1 \tsr \nsbr{N}^{(k-1)}_\mathbf{k'} \cong \nsbr{N}^{1 (k)}_\mathbf{k'}$.  A similar computation shows that $\epsilon_2 \tsr \nsbr{N}^{(k-1)}_\mathbf{k'} \cong \nsbr{N}^{1 (k)}_\mathbf{k'}$.  Also note
$\epsilon_1 \tsr \nsbr{\epsilon}_1 \cong \nsbr{\epsilon}_+$,
$\epsilon_1 \tsr \nsbr{\epsilon}_2 \cong \nsbr{\epsilon}_-$,
$\epsilon_2 \tsr \nsbr{\epsilon}_2 \cong \nsbr{\epsilon}_+$,
and
$\epsilon_2 \tsr \nsbr{\epsilon}_1 \cong \nsbr{\epsilon}_-$.

It remains to prove that the list of irreducibles is distinct, and this is Proposition \ref{p distinct coefficients} (below).
\end{proof}

\begin{corollary}\label{c nsh W semisimple}
The algebra $\field \nsH^{(k)}_W$ is split semisimple for $\field$ the quotient field of  $\mathbf{A^*}$, where
 $\mathbf{A^*}$ is as in Definition \ref{d constants ak}.
\end{corollary}
\begin{proof}
This follows from the proof above, similar to the $W = \S_3$ case.  The base case $k=1$ requires the split semisimplicity of $\field \H_{W}$.  A proof of this is given in \cite[Corollary 8.3.2]{GeckP}.
\end{proof}

\begin{lemma}
\label{l Chebyshev a mathbf k}
For $\mathbf{k}_l,\mathbf{k}_r$ signed $n$-compositions of $k_l$, $k_r$ respectively, there holds
\[ T_{\mathbf{k}_l \mp \mathbf{k}_r} = T_{\mathbf{k}_l} T_{\mathbf{k}_r} \pm T^1_{\mathbf{k}_l} T^1_{\mathbf{k}_r}. \]
\end{lemma}
\begin{proof}
This is immediate from the trigonometric identity for  the cosine of a sum of angles:
\be
\begin{array}{ll}
  T_{\mathbf{k}_l} T_{\mathbf{k}_r} \pm T^1_{\mathbf{k}_l} T^1_{\mathbf{k}_r} \\
  &= \cos(k_{l1} \theta_1 + \dots + k_{ln} \theta_n) \cos(k_{r1} \theta_1 + \dots + k_{rn} \theta_n) \pm \\
& \quad \sin(k_{l1} \theta_1 + \dots + k_{ln} \theta_n) \sin(k_{r1} \theta_1 + \dots + k_{rn} \theta_n)\\
& = \cos((k_{l1} \mp k_{r1})\theta_1+\dots+(k_{ln} \mp k_{rn})\theta_n)\\
& = T_{\mathbf{k}_l \mp \mathbf{k}_r}.
\end{array}
\ee
\end{proof}

\begin{proposition}
\label{p distinct coefficients}
The coefficients $ ( \two^{k} a_{\mathbf{k}'})^{2}$,  $\mathbf{k'} \in \compositions{n}{\leq k}$, together with, if $m$ is even, $(\two^{k} a^{1}_{\mathbf{k'}})^{2}$, $\mathbf{k'} \in \compositions{n}{< k}$, are distinct elements of $\mathbf{A}^{*}$ (see Definition \ref{d constants ak}).
\end{proposition}
\begin{proof}
It suffices to consider the specializations $\sigma :\mathbf{A} \to \RR$, $\sigma(\u) \in \RR_{>0}$. Set $x = \frac{1}{\sigma(\u) + \sigma(\ui)}$; the image of the map $\RR_{>0} \to \RR$, $\sigma(\u) \mapsto x$ is $(0,\frac{1}{2}]$.
If $m$ is odd, it suffices to show that for any signed $n$-compositions $\mathbf{k'}, \mathbf{k''}$,
\[ \cos^{2} (k'_{1} \arccos (\sqrt{w_{1}} x ) + \dots + k'_{n} \arccos(\sqrt{w_{n}} x ) ) = \cos^{2} (k''_{1} \arccos (\sqrt{w_{1}} x ) + \dots + k''_{n} \arccos(\sqrt{w_{n}} x ) )  \]
(an equality of real-valued functions on  $(0,\frac{1}{2}]$) implies $\mathbf{k'}$ is equivalent to $\mathbf{k''}$. Observe that $\cos^{2} (a) = \cos^{2} (b), a, b \in \RR,$ if and only if
\[ a + b \in 2 \pi \ZZ \cup (\pi + 2 \pi \ZZ ) \text{\ \ or \ \ } a-b \in 2 \pi \ZZ \cup (\pi + 2 \pi \ZZ ). \]

Define the functions $\alpha_{\pm} : (0,\frac{1}{2}] \to \RR$ by
\[ x \mapsto (k_{1}' \pm k_{1}'') \arccos (\sqrt{w_{1}} x ) + \dots + (k_{n}' \pm k_{n}'') \arccos (\sqrt{w_{n}} x ) \]
If there exists $x_{0} \in (0,\frac{1}{2})$ such that $\alpha_{+} (x_{0}) \notin 2 \pi \ZZ \cup (\pi + 2 \pi \ZZ )$, then $\alpha_{-} (x)$ is constant in some neighborhood  $N'$ of $x_{0}$.  Otherwise,  $\alpha_+$ is constant on $(0,\frac{1}{2})$. Since we can choose square roots such that $\sqrt{w_{j}} = 2 \cos (\frac{j \pi}{m}) \in (0,2)$, the result follows from the lemma below with  $z_j = 2 \cos (\frac{j \pi}{m})$.  The additional arguments needed in the case $m$ is even are similar.
\end{proof}
We are grateful to Sergei Ivanov for the proof of the following lemma.
\begin{lemma}
\label{l analytic independence}
Suppose $g(x) = \sum_{i =0}^\infty a_i x^i$ is a non-polynomial real analytic function, convergent to this series on a neighborhood $N$ of the origin.  If $z_1,\ldots,z_n$ are distinct positive real numbers, then the functions $1, g(z_{1} x), \dots, g(z_{n} x)$, with domains restricted to  $N'$ for some $N' \subseteq N$ having a limit point in  $N$, are linearly independent over $\RR$.
\end{lemma}
\begin{proof}
Suppose for a contradiction that $c = c_{1} g(z_{1} x) + \dots + c_{n} g(z_{n} x), \ c, c_{j} \in \RR$ with  some $c_{j}$ nonzero.  Assume without loss of generality that all of the $c_{j}$ are nonzero and that $z_1$ is the largest of the $z_j$.  Then
\[ c = \sum_{i=0}^\infty \left( \sum_{j=1}^n c_{j}  z_{j}^{i} \right) a_i x^{i} \]
holds for all $x \in N'$ implies $\sum_{j} c_{j} z_{j}^i = 0$ for all $i > 0$ such that $a_i \neq 0$.  For large $i$, the term $c_1 z_{1}^i$ dominates this sum, hence $c_1 = 0$, contradiction.
\end{proof}
\subsection{}
Now we can determine the nonstandard braid relation for $ \nsH^{(k)}_W$ in the case that $W$ is the dihedral group of order  $2m$.
Define
\be
G_{m,k}(y) =
\begin{cases}
\displaystyle \prod_{\mathbf{k'} \in \compositions{n}{\leq k}} (y - (\two^{k} a_{\mathbf{k'}})^2) & \text{if $m$ is odd}, \\
\displaystyle y \prod_{\mathbf{k'} \in \compositions{n}{\leq k}} (y - (\two^{k} a_{\mathbf{k'}})^2) \prod_{\mathbf{k'} \in \compositions{n}{< k}}(y - (\two^{k} a^1_{\mathbf{k'}})^2) & \text{if $m$ is even}.
\end{cases}
\ee

\begin{theorem}
\label{t W coprod k relations}
Let $W, m,$ and $n$ be as in Theorem \ref{t dihedral irreducibles}.  The algebra $\mathbf{A^*} \nsH^{(k)}_W$ is the associative $\mathbf{A^*}$-algebra generated by $\sP^{(k)}_s, s \in S=\{s_1,s_2\}$, with quadratic relations
\be
\label{e W nsH k quadratic relation}
(\sP^{(k)}_s)^2 = [2]^k \sP^{(k)}_s, \ \ \ s \in S,
\ee
and nonstandard braid relation
\begin{alignat}{2}
\label{e W nsH k braid relation odd}
\sP^{(k)}_1 G_{m, k} (\sP^{(k)}_{21}) &= \sP^{(k)}_2 G_{m,k}(\sP^{(k)}_{12}) & \text{\qquad if $m$ is odd,} \\
\label{e W nsH k braid relation even}
G_{m, k} (\sP^{(k)}_{21}) &= G_{m,k}(\sP^{(k)}_{12}) & \text{\qquad if $m$ is even.}
\end{alignat}
\end{theorem}
\begin{proof}
The proof is similar to that of Theorem \ref{t S3 coprod k relations}.  In the case $m$ is even, we need to check that both sides of \eqref{e W nsH k braid relation even} act on $\nsbr{\epsilon}_{1}$ and $\nsbr{\epsilon}_{2}$ by $0$.  This is clear since $\sP^{(k)}_{12}$ acts on these representations by $0$.
\end{proof}

In fact, we can deduce a stronger statement, which does not seem to be easy to prove directly.
\begin{corollary}
The polynomial $G_{m,k} (y)$ belongs to $\mathbf{A} [y]$.  Therefore Theorem \ref{t W coprod k relations} holds with $\mathbf{A}$ in place of $\mathbf{A^*}. $
\end{corollary}
\begin{proof}
Let $\field $ (respectively  $\field^*$) be the field of fractions of  $\mathbf{A}$ (respectively  $\mathbf{A^*}$).
Let $F$ be the  $\mathbf{A}$-algebra generated by $\sP^{(k)}_s, s \in S,$ with quadratic relations \eqref{e W nsH k quadratic relation}. The ideal of relations $I$ is defined by the exact sequence
\be \label{e exact sequence presentation}
0 \to I \to F \to \nsH^{(k)}_W \to 0.
\ee
Since localizations are flat and free modules are flat, the sequence remains exact after tensoring with  $\field$ and  $\field^*$. Thus Theorem \ref{t W coprod k relations} says that
\[h:=\sP^{(k)}_1 G_{m,k}(\sP^{(k)}_{21}) - \sP^{(k)}_2 G_{m,k}(\sP^{(k)}_{12}) \in \field^* F\]
generates  $\field^* I$ if  $m$ is odd (the  $m$ even case is similar).


Now choose a graded lexicographic term order on monomials in $F$ and
let  $\{g_1,g_2,\ldots\}$ be a Groebner basis for  $\field I$.
There exists an  $i$ such that $g_i$ and $h$ have the same leading monomial.  We must then have $g_i = c h$,  $c \in \field^*$, because if not we could cancel the leading terms of $g_i$ and $h$, contradicting that  $\{h\}$ is a Groebner basis of  $\field^* I$. Since $g_i \in \field F, $ and the leading coefficient of $h$ is  $1$, we must have $c \in \field$.  It follows that  $h \in \field F$ and $\field I= (h)$.  The desired conclusion $I = (h)$ then follows by repeating the argument from the proof of Theorem \ref{t S3 coprod k relations}.
\end{proof}

We may also conclude, as in the $W=S_{3}$ case,
\begin{corollary}
The algebra $\nsH^{(k)}_{W}$ is free as an $\mathbf{A}$-module.
\end{corollary}
\section{A cellular basis for $\nsH^{(k)}_W$}
\label{s cellular basis}
Graham and Lehrer's theory of cellular algebras \cite{GrahamLehrer} formalizes the notion of an algebra with a basis well-suited for studying representations of the algebra.  The theory is modeled after the Kazhdan-Lusztig basis of $\H_r$ in which a basis element $\C_{w}$ is naturally labeled by the insertion and recording tableaux of $w$.  We briefly introduce this theory (following some of the conventions in \cite{Mathas}) and show that $R \nsH^{(k)}_W$ (for $R$ a suitable localization of $\mathbf{A}^{*}$) is a cellular algebra with a cellular basis generalizing the Kazhdan-Lusztig basis of  $\H_3$ and the basis of $\nsH_3$ given in \cite{GCT4}.
\subsection{}
\label{ss cellular basis nsH k}
Let $H$ be an algebra over a commutative ring $R$.
\begin{definition}
\label{d cellular algebra}
Suppose that $(\Lambda, \geq)$ is a (finite) poset and that for each $\lambda \in \Lambda$ there is a finite indexing set $\Tab(\lambda)$ and distinct elements $C^\lambda_{ST} \in H$ for all $S,T \in \Tab(\lambda)$ such that
\[ \Cbasis = \{ C^\lambda_{S T} : \lambda \in \Lambda \text{ and } S,T \in \Tab(\lambda) \} \]
is a (free)  $R$-basis of $H$. For each $\Lambda' \subseteq \Lambda$ let $H_{\Lambda'}$ be the $R$-submodule of $H$ with basis $\{ C^\mu_{S T} : \mu \in \Lambda' \text{ and } S,T \in \Tab(\mu)\}$; write $H_{\lambda}, H_{< \lambda}$ in place of $H_{\{ \lambda \}}$,  $H_{\{ \mu \in \Lambda : \mu < \lambda \}}$.

The triple $(\Cbasis, \Lambda, \Tab)$ is a \emph{cellular basis} of $H$ if
\begin{list}{(\roman{ctr})} {\usecounter{ctr} \setlength{\itemsep}{1pt} \setlength{\topsep}{2pt}}
\item the $R$-linear map $* : H \to H$ determined by $(C^{\lambda}_{S T})^* = C^\lambda_{T S}$, for all $\lambda \in \Lambda$ and all $S$ and $T$ in $\Tab(\lambda)$, is an algebra anti-isomorphism of $H$,
\item for any $\lambda \in \Lambda$ and $h \in H$ there exist $r_{S',S} \in R$, for  $S',S \in \Tab(\lambda)$, such that for all $T \in \Tab(\lambda)$
\[ h C^\lambda_{S T} \equiv \sum_{S' \in \Tab(\lambda)} r_{S', S} C^\lambda_{S' T} \mod H_{< \lambda}. \]
If $H$ has a cellular basis then we say that $H$ is a \emph{cellular algebra}.
\end{list}
\end{definition}
The cellular basis for $R \nsH^{(k)}_W$ ($R$ to be specified) is similar to the ``banal'' example of \cite{GrahamLehrer}, which we now recall.
\begin{example}
\label{ex banal example}
For each element $\lambda \in \Lambda$, we are given an element $\sigma_{\lambda}$ of $R$.  Let $H = R [y] / g (y)$ where $g(y) = \prod_{\lambda \in \Lambda} (y - \sigma_{\lambda})$.  Choose a partial order  $<$ on $\Lambda$ such that for each incomparable pair $\mu , \lambda \in \Lambda$, the element $\sigma_{\mu} - \sigma_{\lambda} \in R$ is invertible (for example, $<$ can be a total order and the $\sigma_\lambda$ can be any elements of any commutative ring $R$).  For $\lambda \in \Lambda$, let $\Tab (\lambda) = \{ \lambda \}$ and set
\[ C_{\lambda, \lambda}^{\lambda} = C^{\lambda} = \prod_{\mu \nleq \lambda} ( y - \sigma_{\mu} ). \]

The triple $(\{C^{\lambda}\}_{\lambda \in \Lambda}, \Lambda, \Tab )$ is a cellular basis of $H$: one checks that $\{C^{\lambda} \}_{\lambda \in \Lambda}$ is an $R$-basis by evaluating a linear relation $\sum_{\lambda \in \Lambda} a_{\lambda} C^{\lambda} = 0$ at $y=\sigma_{\mu}$ for $\mu$ a maximal element of $\Lambda$; one concludes that $a_{\mu} = 0$ and shows by induction that the other $a_{\lambda}$'s are $0$.  Similar considerations show that for any $\lambda \in \Lambda$, $\{ C^{\mu} : \mu < \lambda \}$ is an $R$-basis of the ideal of $H$ generated by $\prod_{\mu \not < \lambda} (y - \sigma_{\mu})$.  We conclude that for $h(y) \in R[y]$, there holds $h(y) C^{\lambda} \equiv h(\sigma_{\lambda}) C^{\lambda} \mod H_{< \lambda}$.
\end{example}

The data for the cellular basis of $R \nsH^{(k)}_{W}$ is as follows:
\begin{itemize}
\item $\Lambda = \Lambda_{1} \cup \Lambda_{2}$, where
\item $\Lambda_{1}=
\begin{cases} \{ \nsbr{\epsilon}_{+}, \nsbr{\epsilon}_{-} \} & \text{if $m$ is odd,} \\
\{ \nsbr{\epsilon}_{+}, \nsbr{\epsilon}_{-}, \nsbr{\epsilon}_{1}, \nsbr{\epsilon}_{2} \} & \text{if $m$ is even,}
\end{cases}$
\item $\Lambda_{2} =
\begin{cases} \compositions{n}{\leq k} & \text{if $m$ is odd,} \\
\compositions{n}{\leq k} \sqcup \compositions{n}{<k}^1 & \text{if $m$ is even,}
\end{cases}$
\item $\Tab(\alpha) = \{ \alpha \}, \alpha \in \Lambda_{1}$,
\item $\Tab(\mathbf{k}) = \{1,2\}, \mathbf{k} \in \Lambda_{2}$.
\end{itemize}
The set $\compositions{n}{<k}^1$ is equal to $\compositions{n}{<k}$ and is decorated with a superscript $1$  to distinguish its elements from those of $\compositions{n}{\leq k}$. It is convenient to define $\sigma_{\mathbf{k}}$ for all $\mathbf{k} \in \Lambda_{2}$ by
\[ \begin{array}{rcll}
\sigma_{\mathbf{k}} &=& \two^{k} a_{\mathbf{k}}, & \mathbf{k} \in \compositions{n}{\leq k},  \\
\sigma_{\mathbf{k}} &=& \two^{k} a^{1}_{\mathbf{k}}, & \mathbf{k} \in \compositions{n}{< k}^1 \text{ and $m$ is even,}
\end{array} \]
where $a_{\mathbf{k}}, a^{1}_{\mathbf{k}}$ are as in Definition \ref{d constants ak}.

Choose a partial order on  $\Lambda_{2}$ such that for each incomparable pair $\mathbf{k},\mathbf{k'} \in \Lambda_{2}$, the constant $\sigma_{\mathbf{k}}^{2} - \sigma_{\mathbf{k'}}^{2}$ is invertible in $R$.  Then let the poset on $\Lambda$ be that of $\Lambda_{2}$ with the elements of $\Lambda_{1}$ added as in the following diagrams.
\be
\label{e lambda poset definition}
\xymatrix@R=3pt@C=8pt{
\nsbr{\epsilon}_{-} \ar @{-} [dd] & & & & & \nsbr{\epsilon}_{-} \ar @{-} [dd] & \\
& & & & & & &\\
\Lambda_{2} \ar @{-} [dd] & & & & & \Lambda_{2} \ar @{-} [dl] \ar @{-} [dr] & \\
& & & & \nsbr{\epsilon}_{1} \ar @{-} [dr] & & \nsbr{\epsilon}_{2} \ar @{-} [dl] \\
\nsbr{\epsilon}_{+} & & & & & \nsbr{\epsilon}_{+} & \\
m \text{ odd} & & & & & m \text{ even}}
\ee


The cellular basis $\Cbasis$ of $R \nsH^{(k)}_{W}$ consists of, if $m$ is odd,
\be
\label{e canonical basis definition m odd}
\begin{array}{l@{\ \ := \ } r}
C^{\nsbr{\epsilon}_-} & \quad 1, \\
C^\mathbf{k}_{1 1} & \sigma_{\mathbf{k}} \sP^{(k)}_1 \prod_{\mathbf{k'} \not \leq \mathbf{k}} (\sP^{(k)}_{21} - \sigma_{\mathbf{k'}}^{2}), \\
C^\mathbf{k}_{2 1} & \sP^{(k)}_{21} \prod_{\mathbf{k'} \not \leq \mathbf{k}} (\sP^{(k)}_{21} - \sigma_{\mathbf{k'}}^{2}), \\
C^\mathbf{k}_{2 2} & \sigma_{\mathbf{k}} \sP^{(k)}_2 \prod_{\mathbf{k'} \not \leq \mathbf{k}} (\sP^{(k)}_{12} - \sigma_{\mathbf{k'}}^{2}), \\
C^\mathbf{k}_{1 2} & \sP^{(k)}_{12} \prod_{\mathbf{k'} \not \leq \mathbf{k}} (\sP^{(k)}_{12} - \sigma_{\mathbf{k'}}^{2}), \\
C^{\nsbr{\epsilon}_+} & \sP^{(k)}_1 \prod_{\mathbf{k'} \not \leq \nsbr{\epsilon}_+} (\sP^{(k)}_{21} - \sigma_{\mathbf{k'}}^{2}),
\end{array}
\ee
for all $\mathbf{k} \in \Lambda_{2}$ and the products are over all $\mathbf{k}' \in \Lambda_{2}$ satisfying the stated conditions.

Note that the quantity defining $C^{\nsbr{\epsilon}_+}$ above is exactly $\sP^{(k)}_1 G_{m, k}(\sP^{(k)}_{21})$, equal to $ \sP^{(k)}_2 G_{m, k}(\sP^{(k)}_{12})$ by the nonstandard braid relation \eqref{e W nsH k braid relation odd}.

If $m$ is even, then the cellular basis $\Cbasis$ consists of
\be
\label{e canonical basis definition even}
\begin{array}{l@{\ \ := \ } r}
C^{\nsbr{\epsilon}_-} & 1, \\
C^\mathbf{k}_{1 1} & \sigma_{\mathbf{k}} \sP^{(k)}_1 \prod_{\mathbf{k'} \not \leq \mathbf{k}} (\sP^{(k)}_{21} - \sigma_{\mathbf{k'}}^{2}), \\
C^\mathbf{k}_{2 1} & \sP^{(k)}_{21} \prod_{\mathbf{k'} \not \leq \mathbf{k}} (\sP^{(k)}_{21} - \sigma_{\mathbf{k'}}^{2}), \\
C^\mathbf{k}_{2 2} & \sigma_{\mathbf{k}} \sP^{(k)}_2 \prod_{\mathbf{k'} \not \leq \mathbf{k}} (\sP^{(k)}_{12} - \sigma_{\mathbf{k'}}^{2}), \\
C^\mathbf{k}_{1 2} & \sP^{(k)}_{12} \prod_{\mathbf{k'} \not \leq \mathbf{k}} (\sP^{(k)}_{12} - \sigma_{\mathbf{k'}}^{2}), \\
C^{\nsbr{\epsilon}_1} & \sP^{(k)}_1 \prod_{\mathbf{k'} \not \leq \nsbr{\epsilon}_1} (\sP^{(k)}_{21} - \sigma_{\mathbf{k'}}^{2}), \\
C^{\nsbr{\epsilon}_2} & \sP^{(k)}_{2} \prod_{\mathbf{k'} \not \leq \nsbr{\epsilon}_2} (\sP^{(k)}_{12} - \sigma_{\mathbf{k'}}^{2}), \\
C^{\nsbr{\epsilon}_+} & \sP^{(k)}_{12} \prod_{\mathbf{k'} \not \leq \nsbr{\epsilon}_+} (\sP^{(k)}_{12} - \sigma_{\mathbf{k'}}^{2}),
\end{array}
\ee
where the products are over all $\mathbf{k'} \in \Lambda_{2}$ satisfying the stated conditions.  The quantity defining $C^{\nsbr{\epsilon}_+}$ above is exactly $G_{m,k}(\sP^{(k)}_{12})$, equal to $G_{m,k}(\sP^{(k)}_{21})$ by the nonstandard braid relation \eqref{e W nsH k braid relation even}.

\begin{proposition}
\label{p cellular basis nsH k}
Let $\Lambda, \Tab, \Cbasis$ and $\sigma_{\mathbf{k}}$ be as above and let $R$ be a ring containing $\mathbf{A^{*}}$ such that $\sigma_{\mathbf{k}}$ is invertible in $R$ for all $\mathbf{k} \in \Lambda_{2}$.  The algebra $H = R \nsH^{(k)}_W$ is a cellular algebra with cellular basis $( \Cbasis, \hat{\Lambda}, \Tab )$.  The anti-automorphism $*$ of Definition \ref{d cellular algebra} (i) is equal to $1^{\text{op}}$ of \textsection\ref{ss theta on nsH k reps}.
\end{proposition}
\begin{proof}
We must show that $\Cbasis$ is an $R$-basis of $H$.  As an $(H, H)$-bimodule, the $(\nsbr{\epsilon}_{+}, \nsbr{\epsilon}_{+})$-isotypic component of the restriction of $H$ to a $(H_{\{ s_{1} \}}, H_{\{ s_{1} \}})$-bimodule has an $R$-basis consisting of monomials of the form $ \sP^{(k)}_{1} ( \sP^{(k)}_{21} )^{j}$.
The set $\{ C^{\mathbf{k}}_{11} : \mathbf{k} \in \Lambda_{2} \} \cup \{ \nsbr{\epsilon}_{+} \}$ (unioned with $\{ \nsbr{\epsilon}_{1} \}$ if $m$ is even) is also an $R$-basis for this space by the same argument as in Example \ref{ex banal example}.
Similar considerations show that for any $\lambda \in \Lambda$ and $i, j \in \{1,2\}$, the set
\[ \{ C^{\mathbf{k'}}_{i j} : \mathbf{k'} < \lambda \} \cup
\begin{cases} \{ \nsbr{\epsilon}_{+}, \nsbr{\epsilon}_{i} \} & \text{if $m$ is even and $i=j$, } \\
\{ \nsbr{\epsilon}_{+} \} & \text{otherwise,}   \end{cases} \]
is an $R$-basis for the $(\nsbr{\epsilon}_{+} , \nsbr{\epsilon}_{+})$-isotypic component of the $(H_{\{ s_{i} \}} , H_{\{ s_{j} \}})$-restriction of $I_{\lambda}$, where $I_{\lambda}$ is the two-sided ideal of $H$ generated by $\prod_{\mathbf{k'} \not < \lambda} (\sP^{(k)}_{12} - \sigma_{\mathbf{k'}}^{2})$.  That $\Cbasis$ is an $R$-basis of $H$ follows by applying this to $\lambda = \nsbr{\epsilon}_{-}$.  We may also deduce that
\begin{multline}
\sP^{(k)}_1 C^{\mathbf{k}}_{21} = \sP^{(k)}_{121} \prod_{\mathbf{k'} \not \leq \mathbf{k}} (\sP^{(k)}_{21} - \sigma_{\mathbf{k'}}^{2}) = \sP^{(k)}_1 (\sP^{(k)}_{21} - \sigma_{\mathbf{k}}^{2} + \sigma_{\mathbf{k}}^{2}) \prod_{\mathbf{k'} \not \leq \mathbf{k}} (\sP^{(k)}_{21} - \sigma_{\mathbf{k'}}^{2}) \\
 = \sP^{(k)}_{1} \prod_{\mathbf{k'} \not < \mathbf{k}} (\sP^{(k)}_{21} - \sigma_{\mathbf{k'}}^{2}) + \sigma_{\mathbf{k}} C^{\mathbf{k}}_{1 1} \equiv \sigma_{\mathbf{k}} C^{\mathbf{k}}_{1 1} \mod H_{< \mathbf{k}}.
\end{multline}
Similarly, $\sP^{(k)}_2 C^{\mathbf{k}}_{12} \equiv \sigma_{\mathbf{k}} C^{\mathbf{k}}_{2 2} \mod H_{< \mathbf{k}}$.  Thus the left action of $H$ on $H_{\leq \mathbf{k}} / H_{< \mathbf{k}}$ in the basis $(C^{\mathbf{k}}_{1 1}, C^{\mathbf{k}}_{2 1}, C^{\mathbf{k}}_{2 2}, C^{\mathbf{k}}_{1 2})$ is given by
\be
\label{e cellular basis action}
\sP^{(k)}_1 \mapsto
\mat{[2]^k & \sigma_{\mathbf{k}} & 0 & 0 \\
0 & 0 & 0 & 0 \\
0 & 0 & 0 & 0 \\
0 & 0 & \sigma_{\mathbf{k}} & [2]^k },\
\sP^{(k)}_2 \mapsto
\mat{0 & 0 & 0 & 0 \\
\sigma_{\mathbf{k}} & [2]^k & 0 & 0 \\
0 & 0 & [2]^k & \sigma_{\mathbf{k}} \\
0 & 0 & 0 & 0}.
\ee
If $m$ is odd, this verifies condition (ii) of Definition \ref{d cellular algebra}.

If $m$ is even, we also need the following, immediate from the definition of $\Cbasis$ in \eqref{e canonical basis definition even},
\[ \begin{array}{rclcrcl}
\sP^{(k)}_{1} \nsbr{\epsilon}_{1} & = & \two^{k} \nsbr{\epsilon}_{1}, & \quad & \sP^{(k)}_{2} \nsbr{\epsilon}_{1} & = & \nsbr{\epsilon}_{+} \equiv \ 0 \mod H_{< \nsbr{\epsilon}_{1}}, \\
\sP^{(k)}_{1} \nsbr{\epsilon}_{2} & = & \nsbr{\epsilon}_{+} \equiv \ 0 \mod H_{< \nsbr{\epsilon}_{2}}, & \quad & \sP^{(k)}_{2} \nsbr{\epsilon}_{2} & = & \two^{k} \nsbr{\epsilon}_{2}.
\end{array}\]

The claim that $* = 1^{\text{op}}$ is straightforward.
\end{proof}

\subsection{}
\label{ss specializations}
Cellular algebras are well-suited for studying specializations.  For this, we need some additional definitions from \cite{GrahamLehrer}.
Let $M_{\lambda}$ be the left $H$-module that is the submodule of $H_{\leq \lambda} / H_{<\lambda}$ with $R$-basis
$\{ C^{\lambda}_{ST} : S \in \Tab (\lambda) \}$ for some $T \in \Tab (\lambda)$; this basis is independent of $T$ and we denote its elements by
$C^{\lambda}_{S} = C^{\lambda}_{ST}, S \in \Tab (\lambda)$.
\begin{definition}
\label{d bilinear form}
For $\lambda \in \Lambda$, the bilinear form $\phi_\lambda : M_{\lambda} \times M_{\lambda} \to R$ is defined in the basis $\{ C^{\lambda}_{S} : S \in \Tab( \lambda ) \}$ as follows. For $S, T \in \Tab(\lambda)$ let
$\phi_{\lambda} (C^\lambda_S, C^\lambda_T )$ be the element of $R$ determined by
\[ C^\lambda_{US} C^\lambda_{TV} \equiv \phi_{\lambda} ( C^\lambda_S, C^\lambda_T ) C^\lambda_{UV} \mod H_{< \lambda}, \]
where $U$ and $V$ are any elements of $\Tab(\lambda)$.
\end{definition}
\begin{proposition}
\label{p bilinear form nsH k}
The bilinear forms $\phi_{\mathbf{k}}, \mathbf{k} \in \Lambda_{2}$, for the cellular algebra $R \nsH^{(k)}_{W}$ of Proposition \ref{p cellular basis nsH k} are given in the basis $(C^{\mathbf{k}}_{1}, C^{\mathbf{k}}_{2}) \times (C^{\mathbf{k}}_{1}, C^{\mathbf{k}}_{2})$ by
\be
\label{e phi sub k}
\Big( \prod_{\mathbf{k'} \not \leq \mathbf{k}} (\sigma_{\mathbf{k}}^{2} - \sigma_{\mathbf{k'}}^{2}) \Big)
\mat{
\two^{k} & \sigma_{\mathbf{k}} \\
\sigma_{\mathbf{k}} & \two^{k}
}.
\ee
\end{proposition}
\begin{proof}
By \eqref{e cellular basis action}, the action of $\sP_{2}^{(k)} ( \sP_{12}^{(k)} )^{j}$ on $M_{\mathbf{k}}$ in the basis $(C^{\mathbf{k}}_{1}, C^{\mathbf{k}}_{2})$ is given by
\[ \mat{
0 & 0 \\
\sigma_{\mathbf{k}}^{2j+1} & \two^{k} \sigma_{\mathbf{k}}^{2j}
}. \]
It follows that $C^{\mathbf{k}}_{2 2} C^{\mathbf{k}}_{21} = \two^{k} \prod_{\mathbf{k'} \not \leq \mathbf{k}} (\sigma_{\mathbf{k}}^{2} - \sigma_{\mathbf{k'}}^{2}) C^{\mathbf{k}}_{21} \mod H_{< \mathbf{k}}$, which accounts for the $(2,2)$ entry of the matrix in \eqref{e phi sub k}. The other entries are computed similarly.
\end{proof}

Now suppose that $R$ is a field.  Let $\rad (M_{\lambda} ) = \{ x \in M_\lambda : \phi_{\lambda} (x,y) = 0 \text{ for all } y \in M_\lambda \}$.  Define $\Lambda' = \{ \lambda \in \Lambda : \phi_{\lambda} \ne 0 \}$.
\begin{proposition}[Graham-Lehrer \cite{GrahamLehrer}]
Maintain the notation of the general setup of Definition \ref{d cellular algebra}.  Let $\lambda \in \Lambda$.  Then
\begin{list}{\emph{(\roman{ctr})}} {\usecounter{ctr} \setlength{\itemsep}{1pt} \setlength{\topsep}{2pt}}
\item $\rad(M_{\lambda})$ is an $H$-submodule of $M_{\lambda}$.
\item If $\phi_{\lambda} \ne 0$, the quotient $L_{\lambda} := M_{\lambda} / \rad(M_{\lambda})$ is absolutely irreducible.
\item If $\phi_{\lambda} \ne 0$, $\rad(M_{\lambda})$ is the minimal submodule of $M_{\lambda}$ with semisimple quotient.
\item The set $\{ L_{\lambda} : \lambda \in \Lambda' \}$ is a complete set of absolutely irreducible $H$-modules.
\end{list}
\end{proposition}

\begin{definition}
The \emph{decomposition matrix} of a cellular algebra $H$ is the matrix $(d_{\lambda \mu})_{\lambda \in \Lambda, \ \mu \in \Lambda'}$, where $d_{\lambda \mu}$ is the multiplicity of $L_{\mu}$ in $M_{\lambda}$.
\end{definition}
We next compute the decomposition matrix of the specialization $\nsH^{(k)}_{W} |_{\u = 1}$.  For $\mathbf{k} \in \Lambda_{2}$, the \emph{residue} of $\mathbf{k}$, denoted $r (\mathbf{k})$, is the unique integer in
\[
\begin{cases}
\{ tm+ \alpha \sum_{j=1}^{n} j k_{j} : t \in \ZZ,\alpha \in \{1,-1\} \} \cap \{ 0, 1, \dots, \lceil {\small \frac{m-1}{2}} \rceil \} & \text{if } \mathbf{k} \in \compositions{n}{\leq k},\\
\{ tm+\frac{m}{2} + \alpha \sum_{j=1}^{n} j k^1_{j} : t \in \ZZ,\alpha \in \{1,-1\} \} \cap \{ 0, 1, \dots, \lceil {\small \frac{m-1}{2}} \rceil \} & \text{if } \mathbf{k^1} \in \compositions{n}{<k}^1.
\end{cases}
\]
It is convenient to define $r(\nsbr{\epsilon}_+) = r(\nsbr{\epsilon}_-) = 0$ and  $r(\nsbr{\epsilon}_1)= r(\nsbr{\epsilon}_2) = \frac{m}{2}.$
\begin{proposition}
\label{p cellular basis u is 1 m odd}
Maintain the notation of Proposition \ref{p cellular basis nsH k}.  Suppose that $m$ is odd, $\Lambda_{2}$ is totally ordered, and  $R = \CC$ with $\mathbf{A}^* \to R$ given by $\u \mapsto 1$. Then for any $\lambda \in \Lambda$,
\[ \phi_{\lambda} \text{ has rank }
\begin{cases}
2 & \text{if $\lambda$ is the maximal element of $\Lambda$ with its residue and $r (\lambda) \ne 0$}, \\
1 & \text{if $\lambda$ is the maximal element of $\Lambda_2$ with residue $0$}, \\
1 & \text{if $\lambda = \nsbr{\epsilon}_-$}, \\
0 & \text{otherwise.}
\end{cases} \]
The decomposition matrix of $\nsH^{(k)}_W |_{\u =1}$ is given by
\[ d_{\lambda \mu } =
\begin{cases}
1 & \text{if $r (\lambda) = r (\mu)$ and  $\{\lambda, \mu\} \ne \{\nsbr{\epsilon}_+,\nsbr{\epsilon}_- \}$,} \\
0 & \text{otherwise.}
\end{cases} \]
\end{proposition}
\begin{proof}
This follows from the computation of  $a_\mathbf{k}|_{\u=1}$,  $a^1_\mathbf{k}|_{\u=1}$ in \eqref{e a coefficients u is 1}, Proposition \ref{p bilinear form nsH k}, \eqref{e cellular basis action}, and \eqref{e bar X distinct}.  In the case $\phi_{\mathbf{k}}$ has rank $1$ $(\mathbf{k} \in \Lambda_2)$, a direct computation shows that $M_{\mathbf{k}}$ has $\nsbr{\epsilon}_{-}$ as an irreducible submodule with quotient $L_{\mathbf{k}} = \nsbr{\epsilon}_{+}$.
\end{proof}

\begin{remark}
Proposition \ref{p cellular basis u is 1 m odd} does not apply to the $\u = 1$ specialization if $m$ is even because $\sigma_{\mathbf{k}}|_{\u=1}$ can be $0$.  It is not difficult to modify the cellular basis to obtain a result similar to Proposition \ref{p cellular basis u is 1 m odd}, which we do below.  This will show that the conclusions of cellular algebra theory do indeed not hold in this case, so the assumptions of Proposition \ref{p cellular basis u is 1 m odd} are necessary, not an artifact of our choice of basis.
\end{remark}

\begin{proposition}
\label{p cellular basis u is 1 m even}
Suppose that $m$ is even and otherwise maintain the setup of Proposition \ref{p cellular basis u is 1 m odd}.  Let $\Cbasis'$ be the same as $\Cbasis$ in \eqref{e canonical basis definition even} with the following modifications
\[
\begin{array}{l@{\ := \ \ }l}
{C'}^\mathbf{k}_{1 1} & \sP^{(k)}_1 \prod_{\mathbf{k'} \not \leq \mathbf{k}} (\sP^{(k)}_{21} - \sigma_{\mathbf{k'}}^{2}), \\
{C'}^\mathbf{k}_{2 2} & \sP^{(k)}_2 \prod_{\mathbf{k'} \not \leq \mathbf{k}} (\sP^{(k)}_{12} - \sigma_{\mathbf{k'}}^{2}). \\
\end{array}
\]
Then $\Cbasis'$ is a free $R$-basis of $H = R \nsH^{(k)}_{W}$.
Let $M_j = \nsbr{X}^{(1)}(w_j)$ as in Theorem \ref{t dihedral irreducibles}.
Then the left-representation afforded by $\{C'^{\mathbf{k}}_{11}, C'^{\mathbf{k}}_{21}\}$
\[    \begin{cases}
    \text{is isomorphic to  $M_{r(\mathbf{k})}|_{\u=1}$} & \text{if $r(\mathbf{k}) \in [n]$},\\
    \text{has a submodule  isomorphic to $\nsbr{\epsilon}_{-}$ with quotient $\nsbr{\epsilon}_{+}$} & \text{if $r(\mathbf{k}) =0$}\\
    \text{has a submodule  isomorphic to $\nsbr{\epsilon}_{2}$ with quotient $\nsbr{\epsilon}_{1}$} & \text{if $r(\mathbf{k}) = \frac{m}{2}$.}
    \end{cases}
    \]
While the left-representation afforded by $\{C'^{\mathbf{k}}_{22}, C'^{\mathbf{k}}_{12}\}$
\[
    \begin{cases}
    \text{is isomorphic to  $M_{r(\mathbf{k})}|_{\u=1}$} & \text{if $r(\mathbf{k}) \in [n]$},\\
    \text{has a submodule  isomorphic to $\nsbr{\epsilon}_{-}$ with quotient $\nsbr{\epsilon}_{+}$} & \text{if $r(\mathbf{k}) =0$}, \\
    \text{has a submodule  isomorphic to $\nsbr{\epsilon}_{1}$ with quotient $\nsbr{\epsilon}_{2}$} & \text{if $r(\mathbf{k}) =\frac{m}{2}$.}
    \end{cases}
    \]
\end{proposition}
\begin{proof}
The proof is similar to that of Propositions \ref{p cellular basis nsH k} and \ref{p cellular basis u is 1 m odd}.
\end{proof}

\begin{corollary}
 The maximal semisimple quotient $ \nsH^{(k)}_W|_{\u=1} / \rad( \nsH^{(k)}_W |_{\u=1})$ is isomorphic to $\CC W$.
\end{corollary}

\section*{Acknowledgments}
I am extremely grateful to Anton Geraschenko for discussions that helped clarify many of the ideas in this paper. I thank Sergei Ivanov for his nice proof of Lemma \ref{l analytic independence} and math overflow for coordinating this help.   I am grateful to Monica Vazirani and Ketan Mulmuley for helpful discussions and to James Courtois and John Wood for help typing and typesetting figures.

\bibliographystyle{plain}
\bibliography{mycitations}

\end{document}